\documentclass[preprint,12pt,nopreprintline]{elsarticle}

\usepackage{mathtools}
\usepackage{amsthm,bm}
\usepackage{amssymb}
\usepackage{enumitem}
\usepackage{graphicx}
\usepackage{booktabs}
\usepackage[utf8]{inputenc}
\usepackage{caption}

\newcommand{\R}{\mathbb{R}}

\newcommand{\ve}{\varepsilon}

\newcommand{\supp}{\operatorname{supp}}

\newcommand{\p}{\partial}
\newcommand{\na}{\nabla}

\newtheorem{thm}{Theorem}[section]
\newtheorem{lemma}[thm]{Lemma}

\newtheorem{pro}[thm]{Proposition}
\newtheorem{cor}[thm]{Corollary}

\newcommand{\x}{\mathbf{x}}
\newcommand{\y}{\mathbf{y}}
\newcommand{\e}{\mathbf{e}}
\newcommand{\n}{\mathbf{n}}
\newcommand{\ds}{\displaystyle}
\newcommand{\C}{\mathbb{C}}
\newcommand{\ri}{\mathrm{i}}

\theoremstyle{definition}
\newtheorem{rmk}[thm]{Remark}

\DeclareMathOperator{\Arg}{Arg}

\begin{document}

\begin{frontmatter}

\title{Cartesian perfectly matched layers for the Helmholtz equation
with variable propagation speed: well-posedness and exponential
convergence}

\author[udea]{C\'esar Ortiz-Saavedra}
\ead{ivan.ortiz@udea.edu.co}

\author[udea]{Mauricio A. Londo\~no-Arboleda\corref{cor1}}
\ead{alejandro.londono@udea.edu.co}

\cortext[cor1]{Corresponding author.}

\address[udea]{Instituto de Matem\'aticas, Facultad de Ciencias Exactas
y Naturales, Universidad de Antioquia, Calle 67 No.\ 53--108,
Medell\'in, Colombia}

\begin{abstract}
We analyse the Cartesian perfectly matched layer (PML) approximation of
the two-dimensional Helmholtz equation
$-c^2\Delta u-\omega^2u=f$ in which the propagation speed $c$ is
\emph{variable} inside the region of interest and equal to a constant
$c_\infty$ outside a disc $B_R$. Existing analyses of the Cartesian PML
assume a homogeneous, or piecewise homogeneous, background. We assume of
$c$ only a two-sided bound, so that discontinuous speeds --- layered
media and compact inclusions --- are covered; this is possible because
the speed enters the PML sesquilinear form only through its zeroth order
coefficient, the principal part being built from the stretching profiles
alone. Well-posedness of the physical problem at this regularity is
obtained from a Lippmann--Schwinger equation, uniqueness resting on
Rellich's lemma and on unique continuation for a bounded potential. We
also show that a Morawetz multiplier yields a sharp non-trapping
condition on $c$, namely $\na c(\x)\cdot(\x-\x_0)<c(\x)$ for some
$\x_0$, which reduces in the radial case to the classical Herglotz
condition $\frac{d}{d\varrho}(\varrho/c(\varrho))>0$: its failure
produces a trapped bicharacteristic. We then show that the variable
speed is invisible to the essential spectrum: after factoring out $c^2$, the two operators differ
by a compactly supported multiplication, so the spectral analysis of the
homogeneous Cartesian layer applies verbatim. This yields
well-posedness of the truncated problem on a rectangle $\Omega_\delta$
for \emph{every} $\omega>0$, with a stability constant independent of
the layer width, and convergence to the physical solution in $H^1(\Omega)$ at a rate
exponential in $\omega\gamma_M\delta_M$, where $\gamma_M$ and
$\delta_M$ are the absorption strength and the layer thickness. The convergence proof rests on two ingredients of
independent interest: uniform bounds for $H^{(1)}_0$ and $H^{(1)}_1$ on
the closed first quadrant, valid up to the boundary rays where the
stretching degenerates, and a geometric lemma showing that the imaginary
part of the complexified distance associated with the \emph{separable}
Cartesian stretching is non-negative throughout the layer and bounded
below away from the interface. Numerical experiments with a low-velocity inclusion
confirm the predicted exponential decay in both $\delta$ and $\omega$.
\end{abstract}

\begin{keyword}
Helmholtz equation \sep perfectly matched layer \sep Cartesian PML \sep
variable propagation speed \sep non-trapping \sep exponential convergence
\MSC[2020] 35J05 \sep 47A10 \sep 35P25 \sep 65N12
\end{keyword}

\end{frontmatter}

\section{Introduction}
\label{sec:intro}

Time-harmonic wave propagation in an unbounded medium must be truncated
before it can be computed. Among the available devices --- absorbing
boundary conditions~\cite{EngquistMajda1977,EngquistMajda1979,BaylissTurkel1980},
exact non-reflecting conditions~\cite{HagstromKeller1986,KellerGivoli1989},
boundary integral formulations and infinite elements --- the perfectly
matched layer introduced by B\'erenger~\cite{Berenger1994,Berenger1996}
has become the most widely used, because it is local, easy to
discretise, and, in exact arithmetic, reflectionless.

The modern way to read a PML is as a complex change of variables: the
solution is continued analytically into a complex neighbourhood of the
real domain, along which an outgoing wave decays exponentially, and the
computational domain is then truncated with any convenient boundary
condition. This point of view was put on a rigorous footing by Lassas
and Somersalo~\cite{LassasSomersalo1998,LassasSomersalo2001} for
spherical and, more generally, smooth convex layers, and by Bramble and
Pasciak~\cite{BramblePasciak2007,BramblePasciak2008} for finite spherical
layers in acoustics and electromagnetics. Quantitative and adaptive
variants, with constants explicit in the wave number, were obtained by
Chen and Liu~\cite{ChenLiu2005}; the influence of the profile
$\sigma$ and of the layer geometry has been studied by Collino and
Monk~\cite{CollinoMonk1998,CollinoMonkCurvi1998} and by Berm\'udez et
al.~\cite{Bermudez2007}, and PML has been used to compute resonances of
open systems~\cite{HeinHohageKoch2004,HeinHohageKochSchoberl2007,KimPasciak2009}.

Cartesian layers --- those obtained by stretching each Cartesian
coordinate separately --- are by far the most convenient in practice,
since they fit the rectangular meshes used in finite difference and
finite element codes, but they are markedly harder to analyse than
radial ones. The reason, identified by Kim and
Pasciak~\cite[\S1]{KimPasciak2010b}, is that the standard argument ---
coercivity up to a compact perturbation --- breaks down: for a Cartesian
layer the region in which the sesquilinear form fails to be coercive is
\emph{unbounded}, so the offending term cannot be absorbed as a lower
order perturbation on a bounded set. Kim and Pasciak resolved this for a
homogeneous background by first analysing the essential spectrum of the
Cartesian PML operator on the whole plane~\cite{KimPasciak2010a} and
then deducing well-posedness of the infinite-domain problem for every
real $k\ne0$~\cite[Lem.~4.6, Thm.~4.7]{KimPasciak2010b}, from which
stability and exponential convergence of the truncated problem follow.
Bramble and Pasciak~\cite{BramblePasciak2013} obtained, for the same
homogeneous background, inf--sup conditions on $\R^n$ and on the
truncated square with constants independent of the layer width and of
the absorption strength, in $n=2$ and $n=3$. Duan, Jiang and
Zheng~\cite{DuanJiangZheng2020} extended the Cartesian analysis to
two-layer media for Maxwell's equations, and Lu, Lai and
Wu~\cite{LuLaiWu2024} obtained, for a two-layer acoustic medium with a
flat interface, well-posedness of the uniaxial PML problem
unconditionally in the layer width. For a \emph{radial} layer the
high-frequency question has since been settled in considerable
generality by Galkowski, Lafontaine and
Spence~\cite{GalkowskiLafontaineSpence2023}, who prove that PML
truncation is exponentially accurate with constants explicit in the
frequency; the Cartesian counterpart of that statement remains
fragmentary, and we return to it in Remark~\ref{rmk:open}. Where a genuinely variable speed has been admitted for a Cartesian
layer, it has come with a structural hypothesis: smoothness
in~\cite{GalkowskiGongGrahamLafontaineSpence2024}, where $c\in
C^\infty$ is non-trapping and the truncation error is
$O(k^{-\infty})$, and periodicity together with a separation of scales
in the heterogeneous media of Jiang, Sun, Sun and
Ma~\cite{JiangSunSunMa2024}, where exponential convergence is recovered
through homogenisation. The semiclassical machinery behind much of this
progress is surveyed in~\cite{GalkowskiSpence2026}.

What every one of these analyses requires, then, is that the medium be
homogeneous in the region of interest, or else inhomogeneous in some
prescribed way: two layers, a periodic microstructure, a smooth
profile. The media that motivate domain truncation in seismology, in
medical imaging and in non-destructive testing satisfy none of these
descriptions. Their defining feature is a speed that jumps across
interfaces of unprescribed shape and number --- it is the reflection
and refraction produced by those jumps that one is trying to compute
--- and between the interfaces the speed need be no better than
bounded.

This paper analyses the Cartesian PML for exactly that class. Of the
propagation speed we assume only that it is measurable, that it obeys a
two-sided bound $0<c_{\min}\le c\le c_{\max}$, and that it is constant
outside a disc; nothing else. Piecewise constant speeds are therefore
admissible, and so is any bounded measurable profile between the
interfaces. That such a weak hypothesis suffices is not an accident of
the estimates but a structural feature of the Cartesian layer, and
Section~\ref{subsec:contributions} begins by explaining it.

\subsection{Contributions}
\label{subsec:contributions}

Throughout we consider $-c^2\Delta u-\omega^2u=f$ in $\R^2$ with the
Sommerfeld radiation condition, under the hypotheses on $c$ just
described, and with $f$ supported in $\overline{B_R}$. Our
contributions are the following.

\begin{enumerate}[label=(\roman*),leftmargin=2.2em]
\item \emph{A theory that does not see the regularity of the speed.}
The reason a two-sided bound suffices is structural: the propagation
speed enters the Cartesian PML sesquilinear form only through the
coefficient $J=c^{-2}d_1d_2$ of its zeroth order term, the principal
part being built from the stretching profiles $\sigma_1,\sigma_2$
alone, and those we choose ourselves. No estimate below ever
differentiates $c$. Well-posedness of the physical problem at this
regularity is Proposition~\ref{prop:physical}: uniqueness from
Rellich's lemma and unique continuation for a bounded potential,
existence from the Lippmann--Schwinger equation.

\item \emph{A sharp non-trapping condition, for a differentiable speed.}
The next contribution is stated in a smaller class than the rest of the
paper, and it is worth saying at once why, since the two hypotheses do
not sit at the same level. Everything in (i) concerns well-posedness,
and for that a two-sided bound is enough. Nothing in (i), however, says
how the solution behaves as $\omega\to\infty$, and that question is
geometric: it is decided by whether the rays of the medium escape to
infinity or return. Rays require a differentiable speed to be defined
at all. Proposition~\ref{prop:non_trapping} therefore assumes
$c\in C^{1,1}(\R^2)$ --- and is used in the proof of no other statement
in this paper.

Within that class we show that
\begin{equation}\label{eq:intro_H3}
  \na c(\x)\cdot(\x-\x_0) < c(\x)\qquad\text{for all }\x\in\R^2
\end{equation}
and some $\x_0\in\R^2$ forces every bicharacteristic of $c|\xi|$ to
escape. The Morawetz multiplier used is the flat one,
$(\x-\x_0)\cdot\xi$; what makes~\eqref{eq:intro_H3} sharp rather than
merely sufficient is that the two terms of $\dot M$ must be kept
together, so that the quantity which controls the sign is the
scale-invariant ratio $c^{-1}\na c\cdot(\x-\x_0)$. In the radial
case~\eqref{eq:intro_H3} is exactly the Herglotz condition
$\frac{d}{d\varrho}\bigl(\varrho/c(\varrho)\bigr)>0$ of geometrical
seismology, whose failure is the circular-orbit condition; see
Remark~\ref{rmk:sharpness} and Figure~\ref{fig:sharpness}. In
particular~\eqref{eq:intro_H3} admits low-velocity inclusions, which the
naive requirement $\na c\cdot(\x-\x_0)\le0$ excludes entirely.

The gap between the two classes is not an artefact of our proof. A
piecewise constant speed with a low-velocity inclusion traps a
positive-measure set of rays by total internal reflection, so no
non-trapping statement can hold at the regularity of (i); what is lost
there is not well-posedness but the control of the solution operator
at high frequency. Remark~\ref{rmk:mollify} makes this precise, and
Remark~\ref{rmk:open} records what it costs.

\item \emph{Well-posedness of the truncated Cartesian PML problem.}
The truncated operator $T_\delta$ on the rectangle $\Omega_\delta$ has
compact resolvent (Proposition~\ref{prop:compact_resolvent}), hence
discrete spectrum (Corollary~\ref{cor:discrete_spectrum}), and the
truncated problem is well posed for every $\omega$ outside an at most
countable set (Proposition~\ref{prop:truncated_wellposed}).

\item \emph{A localisation lemma, and well-posedness for every
frequency.} Lemma~\ref{lem:reduction} shows that
$-c^2\widetilde\Delta$ and $-c_\infty^2\widetilde\Delta$ have the same
essential spectrum. The point is that subtracting the two operators
produces a second order operator with compactly supported coefficients,
which is \emph{not} relatively compact; factoring out $c^2$ first moves
the compactly supported factor to the zeroth order term, where
Rellich--Kondrachov applies. The essential spectrum of the homogeneous
Cartesian layer, computed in~\cite{KimPasciak2010a}, is a sector
meeting $[0,\infty)$ only at the origin, so
$-c^2\widetilde\Delta-\omega^2$ is Fredholm of index zero for every
$\omega>0$; injectivity (Proposition~\ref{prop:injective}) then makes
it invertible, and a reflection argument transfers this to
$\Omega_\delta$ with constants independent of the layer width
(Proposition~\ref{prop:uniform}). The exceptional set of frequencies
therefore disappears.

\item \emph{Uniform Hankel bounds on the closed first quadrant.}
Lemma~\ref{lem:hankel} gives
$|H^{(1)}_0(\zeta)|\le\sqrt{2/(\pi|\zeta|)}\,e^{-\Im\zeta}$, and an
analogous bound for $H^{(1)}_1$, for $0\le\Arg\zeta\le\pi/2$. The bounds
are sharp and, crucially, uniform \emph{up to} the boundary rays: this
is what makes them usable at the interface $\p\Omega$, where the
stretching degenerates.

\item \emph{A geometric lemma for the separable stretching.}
Lemma~\ref{lem:geometry} shows that the complexified distance
$\rho(T(\x)-\y)$ has non-negative imaginary part throughout the layer,
and gives an explicit positive lower bound for it away from the
interface. This replaces the mode-by-mode continuation of a
Fourier--Bessel series, which is natural for a radial PML but
unavailable here: the Cartesian map does not preserve polar coordinates,
so $T(\x)$ corresponds to no single complex value of $r$
(Remark~\ref{rmk:no_fourier_bessel}).

\item \emph{Exponential convergence.}
Theorem~\ref{thm:convergence} and Corollary~\ref{cor:exp_rate} give
$\|u_\delta-u\|_{H^1(\Omega)}\le C(\omega,\delta)\,
e^{-c\omega\gamma_M\delta_M}$, in $H^1(\Omega)$ rather than
$L^2(\Omega)$; by Corollary~\ref{cor:every_omega} this holds for every
$\omega>0$ with a constant independent of $\delta$.
Section~\ref{sec:numerics} confirms the predicted behaviour numerically
on a medium with a $40\%$ low-velocity inclusion.
\end{enumerate}

\subsection{Relation to the homogeneous-background theory}
\label{subsec:comparison}

It is worth stating precisely how the present results relate to those of
Kim and Pasciak~\cite{KimPasciak2010a,KimPasciak2010b} for a homogeneous
background, since we lean on them heavily.

Their analysis has two halves. The first is spectral: the essential
spectrum of the Cartesian PML operator on the whole plane is a sector
meeting the real axis only at the origin~\cite{KimPasciak2010a}. The
second is analytic: an integral representation of the stretched solution
yields uniqueness, and hence stability for every real
$k\ne0$~\cite{KimPasciak2010b}. Both halves are proved for
$c\equiv c_\infty$.

Our contribution is to show that the first half survives a variable
speed unchanged, and that the second half survives it with the physical
uniqueness statement replaced by one valid for a variable, and merely
bounded, coefficient.
Lemma~\ref{lem:reduction} does the first: the essential spectrum cannot
see a perturbation supported in $\overline{B_R}$ once the operator is
factored so that the perturbation sits at zeroth order. The second is
Proposition~\ref{prop:injective}, where the representation formula is
applied on $\{|\x|>R\}$ --- a region in which the medium \emph{is}
homogeneous, so that~\cite[Thm.~4.3]{KimPasciak2010b} applies verbatim
--- and the gluing is closed by Rellich's lemma and unique
continuation. The reflection argument
of~\cite[Thm.~4.8]{KimPasciak2010a} then transfers stability to
$\Omega_\delta$, its only requirement being that the coefficients be
constant near $\p\Omega_\delta$, which~\ref{hyp:saturate} guarantees.
In both halves the variable speed is confined to a zeroth order term,
which is why the transfer costs no regularity: the smoothness
hypothesis of~\cite{KimPasciak2010a,KimPasciak2010b} on the background
is inherited by $\sigma_1,\sigma_2$, not by $c$.

What is \emph{not} obtained here is uniformity in $\omega$: the
threshold $\delta_0(\omega)$ and the constant $C_\star(\omega)$ of
Proposition~\ref{prop:uniform} are not shown to be bounded as
$\omega\to\infty$. This limitation is inherited from the qualitative
nature of the Fredholm argument rather than caused by the variable
speed. It should be said that the state of the art has moved here:
for a radial layer the frequency-explicit theory is now
complete~\cite{GalkowskiLafontaineSpence2023}, and for a Cartesian
layer with a constant speed Wang and Zheng~\cite{WangZheng2026} have
obtained an inf--sup constant of the optimal order $O(k^{-1})$, with a
threshold on $\sigma_0\delta$ that is independent of $k$, by estimating
the stretched Green kernel directly. What is still missing is the
combination of a Cartesian layer with a variable speed, which is
exactly the configuration treated here; Remark~\ref{rmk:open} explains
why a perturbation argument does not bridge the two. The
inhomogeneity of $c$ does, however, enter the size of the error, through
the Cauchy data $\mathcal N_R(u)$ of the physical solution on $\p B_R$;
the exponential rate itself is produced entirely in the homogeneous
exterior and is therefore the same as in the constant-speed case.

\subsection{Outline}

Section~\ref{sec:helmholtz} states the problem and the standing
hypotheses. Section~\ref{sec:preliminaries} introduces the Cartesian PML
formulation, the pointwise bounds on the stretching coefficients, the
well-posedness of the extended problem, the well-posedness of the
physical problem, and the non-trapping proposition. Section~\ref{sec:truncated} treats the truncated problem by elementary
means. Section~\ref{sec:allomega} removes the exceptional set of
frequencies and the dependence of the stability constant on the layer
width. Section~\ref{sec:convergence} contains the convergence theorem.
Section~\ref{sec:numerics} reports numerical experiments and
Section~\ref{sec:conclusions} draws conclusions.
\ref{app:Ltilde} records elementary operator-theoretic properties of
the extended-domain PML operator which are of independent interest.

\section{The variable-speed Helmholtz problem}
\label{sec:helmholtz}

Throughout, $c$ is a real-valued function on $\R^2$ subject to the two
standing hypotheses
\begin{enumerate}[label=(H\arabic*),ref=(H\arabic*)]
  \item\label{hyp:support} there exist constants $c_\infty>0$ and $R>0$
        such that $c(x,y)=c_\infty$ for $\sqrt{x^2+y^2}>R$;
  \item\label{hyp:positive} $c$ is measurable and there are constants
        $c_{\text{min}},c_{\text{max}}>0$ with
        $0<c_{\text{min}}\le c(x,y)\le c_{\text{max}}$; note that
        $c_{\text{min}}\le c_\infty\le c_{\text{max}}$,
\end{enumerate}
both understood almost everywhere. No regularity of $c$ beyond
\ref{hyp:positive} is assumed in the main line of the paper. In particular $c$
may be discontinuous, so that the layered media and compact inclusions
of applications, whose speeds are piecewise constant, are admissible.
The single exception is Proposition~\ref{prop:non_trapping}, which
characterises the non-trapping regime and is stated separately for
$c\in C^{1,1}(\R^2)$; it is used in the proof of no other statement, as
Remark~\ref{rmk:no_ntrap_needed} records. (\ref{app:Ltilde} contains,
in addition, an auxiliary construction that needs a Lipschitz speed;
this is stated there and is not used in the body of the paper.)

Given $\omega>0$ and $f\in L^2(\R^2)$ with compact support, we look for
the outgoing solution of
\begin{equation}\label{eq:helmholtz}
  \begin{cases}
    -c^2\Delta u-\omega^2u=f & \text{in }\R^2,\\[2pt]
    \ds\lim_{r\to\infty}r^{1/2}
      \left(c\,\frac{\p u}{\p r}-\ri\omega u\right)=0 .
  \end{cases}
\end{equation}
Since $c\equiv c_\infty$ near infinity by~\ref{hyp:support}, the
radiation condition is the classical Sommerfeld condition at wavenumber
$k=\omega/c_\infty$. Well-posedness of~\eqref{eq:helmholtz}
under~\ref{hyp:support}--\ref{hyp:positive} alone is the content of
Proposition~\ref{prop:physical}.

The absorbing layer will be built on a rectangle
$\Omega=(-a,a)\times(-b,b)$, introduced in
Section~\ref{sec:preliminaries}. Starting in
Section~\ref{sec:allomega} we shall additionally require
\begin{equation}\label{eq:hyp_R}
  R<\min\{a,b\},
\end{equation}
so that the closed disc $\overline{B_R}=\{\x\in\R^2:|\x|\le R\}$ is
compactly contained in $\Omega$: all the inhomogeneity of the medium is
then confined to $\Omega$, and $c\equiv c_\infty$ on a neighbourhood of
$\p\Omega$, throughout the layer and beyond. This is the standard
configuration for PML truncation, and it is what allows the analytic
continuation of Section~\ref{sec:convergence}. The results of
Sections~\ref{sec:preliminaries} and~\ref{sec:truncated} do not
use~\eqref{eq:hyp_R}.

\section{The Cartesian PML formulation and basic estimates}
\label{sec:preliminaries}

To establish the PML-Helmholtz equation within a rectangular domain, we define some fundamental functions. Let $\sigma:\mathbb{R}\rightarrow\mathbb{R}$ be defined as follows:
\begin{equation}\label{eq:PML_thickness}
	\sigma(t;l,\delta,\gamma)=\begin{cases}
		0 &\mbox{ for } |t|\leq l\\
		\gamma p\left(\frac{|t|-l}{\delta}\right) &\mbox{ for } l< |t|< l+\delta\\
		\gamma & \mbox{ for } |t|\geq l+\delta
	\end{cases}
\end{equation}
where $\delta$, $l$ and $\gamma$ are positive constants, and $p:(0,1)\rightarrow (0,1)$ is an increasing function, with $p(0^+)=0$ and $p(1^-)=1$, such that $\sigma\in C^\infty(\R)$. One such increasing function is $p(t)=1-\exp\left(\frac{t^2}{t^2-1}\right)$.
We note that the shape hypothesis is imposed here on $\sigma_j$ itself,
whereas in~\cite{KimPasciak2010a,KimPasciak2010b} it is imposed on the
averaged profile $\tilde\sigma_j$. The two conventions are related by
$\sigma_j=(t\tilde\sigma_j)'$, and what those references actually use is
that $d_j=1+\ri(t\tilde\sigma_j)'=1+\ri\sigma_j$ be constant outside a
compact set; this holds here, since $\sigma_j\equiv\gamma_j$ there, even
though $\tilde\sigma_j$ only tends to $\gamma_j$.

\subsection{Notation}
\begin{itemize}
    \item $\Omega=(-a,a)\times(-b,b)$
    \item $\Omega_\delta=(-a-\delta_a,a+\delta_a)\times(-b-\delta_b,b+\delta_b)$
    \item $\sigma_1(x)=\sigma(x\,;a,\delta^{\mathrm t}_a,\gamma_a)$ and $\sigma_2(y)=\sigma(y\,;b,\delta^{\mathrm t}_b,\gamma_b)$, where the \emph{transition widths} $\delta^{\mathrm t}_a,\delta^{\mathrm t}_b>0$ are fixed once and for all, independently of the layer widths $\delta_a,\delta_b$; we always assume $\delta^{\mathrm t}_a\le\delta_a$ and $\delta^{\mathrm t}_b\le\delta_b$, so that the absorption has reached its full strength $\gamma_a$, resp.\ $\gamma_b$, before the outer boundary is met
    \item $\displaystyle\tilde{\sigma}_j(t)=\frac{1}{t}\int_0^t\sigma_j(s)ds$
    \item $\sigma_M=\max\left\{\gamma_a,\gamma_b\right\}$.
    \item $U=\left\{z\in\C\ :\ \Re(z)>-\frac{1}{2\sigma_M} \right\}$
    \item For $z\in U$, $\tilde{d}_j^z(t)=1+z\tilde{\sigma}_j(t)$
    \item $\tilde{x}^z(x) = x \tilde d_1^z(x)$ and $\tilde{y}^z(y) = y \tilde d_2^z(y)$
    \item $T^z(x,y)=(\tilde{x}^z(x),\tilde{y}^z(y))$
    \item $\displaystyle d_1^z(x)=\frac{d\tilde x^z}{ dx}=1+z\sigma_1(x)$ and $\displaystyle d_2^z(y)=\frac{d\tilde y^z}{ dy}=1+z\sigma_2(y)$.
    \item $s_1^z(x)=\frac{1}{d_1^z(x)}$ and $s_2^z(y)=\frac{1}{d_2^z(y)}$
    \item $\tilde r^z=\|T^z(x_2,y_2)-T^z(x_1,y_1) \|_2$
    \item $J^z(x,y)=c(x,y)^{-2}d_1^z(x)d_2^z(y)$
    \item $\tilde \Delta^z=\frac{1}{d_1^z(x)}\p_x \left(\frac{1}{d^z_1(x)}\p_x \right)+\frac{1}{d^z_2(y)}\p_y \left(\frac{1}{d^z_2(y)}\p_y \right)$
    (When $z=\ri=\sqrt{-1}$, we omit the superscript $z$).
    \item $\displaystyle A_W(u,v)=\int_{W}\left(\frac{d_2(y)}{d_1(x)}\p_x u\p_x \bar{v}+\frac{d_1(x)}{d_2(y)}\p_y u\p_y \bar{v}\right)dxdy$
    \item $\displaystyle(u,v)_W=\int_{W}u\bar{v}dxdy$
    \item For $w\in\C$,
    \begin{equation}\label{eq:PMLbilinear}
        \displaystyle\mathcal{A}_{W,w}(u,v)=A_W(u,v)-w(Ju,v)_W.
    \end{equation}
 \end{itemize}

\textbf{Extended problem}. Given $f \in L^2(\R^2)$ and $w\in \C$, find $u \in H^1(\R^2)$ such that for any $v \in H^1(\R^2)$,
\begin{equation}\label{eq:PML_extended_problem}
\mathcal{A}_{\R^2,w}(u,v) = (Jf,v)_{\R^2}.
\end{equation}

\textbf{Truncated problem}. Given $f \in L^2(\Omega_\delta)$ and $w\in \C$, find $u \in H^1_0(\Omega_\delta)$ such that for any $v \in H^1_0(\Omega_\delta)$,
\begin{equation}\label{eq:PML_truncated_problem}
\mathcal{A}_{\Omega_\delta,w}(u,v) = (Jf,v)_{\Omega_\delta}.
\end{equation}

\begin{lemma}\label{lemma:bounds1}
	The following bounds are satisfied for any $(x,y) \in \R^2$.
\begin{enumerate}
    \item  $\ds \left|\frac{d_1(x)}{d_2(y)}\right|\leq\sqrt{\sigma_M^2+1} $ and $\ds  \left|\frac{d_2(y)}{d_1(x)}\right|\leq\sqrt{\sigma_M^2+1}$.
     \item $\ds |c^{-2}(x,y)d_1(x)d_2(y)|\leq c_\text{min}^{-2}\left(\sigma_M^2+1\right)$
    \item $\ds \Re\left(\frac{d_1(x)}{d_2(y)}\right)\geq\frac{1}{\sigma_M^2+1}>0$ and  $\ds \Re\left(\frac{d_2(y)}{d_1(x)}\right)\geq\frac{1}{\sigma_M^2+1}>0.$
    \item Besides, if $\gamma_a\leq\gamma_b= \sigma_M<1$ and $\Arg\left(1+\ri\gamma_a\right)\leq\theta\leq\Arg\left(1+\ri\gamma_b\right)$, then
    \begin{equation}
	\Re\left(e^{-\ri\theta}d_1(x)d_2(y)\right)\geq \frac{1-\sigma_M^2}{\sqrt{\sigma_M^2+1}}>0.
	\end{equation}
\end{enumerate}
\end{lemma}
\begin{proof} By direct calculations we have
\begin{enumerate}
    \item
        $\ds \left|\frac{d_1(x)}{d_2(y)}\right|=\sqrt{\frac{\sigma _1^2(x)+1}{\sigma _2^2(y)+1}}\leq\sqrt{{\sigma_M^2+1}{}} $
        \item $\ds c_\text{max}^{-2}\le\left|c^{-2}(x,y)d_1(x)d_2(y)\right|
          =c^{-2}(x,y)\sqrt{\left(\sigma _1^2(x)+1\right)
            \left(\sigma _2^2(y)+1\right)}
          \leq c_\text{min}^{-2}\left(\sigma_M^2+1\right)$
\item $\ds \Re \left( \frac{d_1(x)}{d_2(y)} \right)=\frac{\sigma _1(x) \sigma _2(y)+1}{\sigma _2^2(y)+1}\geq\frac{1}{\sigma_M^2+1}$
    \item Let $\gamma$ be in $[\gamma_a,\gamma_b]$. Thus, with $\theta$ such that $\cos\theta=\frac{1}{\sqrt{1+\gamma^2}}$ and $\sin\theta=\frac{\gamma}{\sqrt{1+\gamma^2}}$, we have
    \begin{equation}
        \Re\left(e^{-\ri\theta} d_1(x)d_2(y)  \right)=\frac{1 +\gamma(\sigma_1(x)+\sigma_2(y))-\sigma _1(x)\sigma_2(y) }{\sqrt{\gamma ^2+1}}\geq\frac{1-\sigma_M^2}{\sqrt{\sigma_M^2+1}}
    \end{equation}
\end{enumerate}
\end{proof}

\begin{lemma}\label{lm:well-possed}
	If \ $0<\gamma_a\leq\gamma_b=\sigma_M<1$, then for $w=-e^{-\ri\theta}$, with $\Arg\left(1+\ri\gamma_a\right)\leq\theta\leq\Arg\left(1+\ri\gamma_b\right)$, the equation \eqref{eq:PML_extended_problem} defines a well-posed problem.
\end{lemma}
\begin{proof} Actually, we can even prove that given $F\in H^{-1}(\R^2)$ there exists a unique function $u\in H^1(\R^2)$ such that for any $v\in H^1(\R^2)$ it satisfies $\mathcal{A}_{\R^2,w}(u,v)=F(v)$ and there exists a positive constant $C$ such that $\|u\|_{H^1}\leq\frac{1}{C}\|F\|_{H^{-1}}$. To get the result
	we need to show that there exist positive constants $C_1$ and $C_2$ such that for all $u,v\in H^1(\R^2)$ the sesquilinear form $\mathcal{A}_{\R^2,w}$ satisfies
	\begin{equation}
	|\mathcal{A}_{\R^2,w}(u,v)|\leq C_1\|u\|_{H^1}\|v\|_{H^1}  \mbox{ (continuity)}
	\end{equation}
	and
	\begin{equation}
	\Re\,\mathcal{A}_{\R^2,w}(u,u)\geq C_2 \|u\|^2_{H^1},  \mbox{ (coercivity) }
	\end{equation}
	once done this Lax-Milgram's theorem is applied.
	The continuity of $\mathcal{A}_{\R^2,w}$ is easily verified by using the uniform bounds in Lemma~\ref{lemma:bounds1}. Indeed,
	by taking
 \begin{equation}
     \alpha=\max\left\{\sqrt{\sigma_M^2+1} , c_\text{min}^{-2}|w|\left(\sigma_M^2+1\right)\right\}
 \end{equation}
   and then by H\"older's inequality
	\begin{equation}
     \begin{split}
      |\mathcal{A}_{\R^2,w}(u,v)|&\leq\alpha\int_{\R^2}\left(\left|\p_x u\p_x \bar{v}\right| + \left|\p_y u\p_y \bar{v}\right|+|u\bar{v}|\right)dxdy\\
	&\leq\alpha\left(\left\|\p_x u\right\|_2\left\|\p_x v\right\|_2+\left\|\p_y u\right\|_2\left\|\p_y v\right\|_2 + \|u\|_2\|{v}\|_2\right)\\
	&\leq C_1\|u\|_{H^1}\|v\|_{H^1},
     \end{split}
	\end{equation}
	where $C_1=3\alpha$.

	Recall that $c$ is real valued and satisfies the two-sided bound $0 < c_\text{min} \leq c(x,y) \leq c_\text{max}$ almost everywhere; no smoothness of $c$ is used here, nor anywhere in this section. Thus, taking
	\begin{equation}
	C_2=\min\left\{\frac{1}{\sigma_M^2+1},\, c_\text{max}^{-2}\frac{1-\sigma_M^2}{\sqrt{\sigma_M^2+1}}\right\}
	\end{equation}
	we have
\begin{equation}
     \begin{split}
  \Re\, \mathcal{A}_{\R^2,w}(u,u)&=\Re\int_{\R^2}\Bigl(\frac{d_2(y)}{d_1(x)}\left|\p_x u\right|^2+\frac{d_1(x)}{d_2(y)}\left|\p_y u\right|^2\Bigr)dxdy\\
      &\qquad+\Re\int_{\R^2}e^{-\ri\theta}c(x,y)^{-2}d_1(x)d_2(y)|u|^2\,dxdy\\
      &=\int_{\R^2}\Bigl(\Re\Bigl(\frac{d_2(y)}{d_1(x)}\Bigr)\left|\p_x u\right|^2+\Re\Bigl(\frac{d_1(x)}{d_2(y)}\Bigr)\left|\p_y u\right|^2\Bigr)dxdy\\
      &\qquad+\int_{\R^2}\Re\Bigl( e^{-\ri\theta}c(x,y)^{-2}d_1(x)d_2(y)\Bigr)|u|^2\,dxdy\\
	&\geq C_2\int_{\R^2}\left(\left|\p_x u\right|^2+\left|\p_y u\right|^2+|u|^2\right)dxdy\\
	&=C_2\|u\|^2_{H^1}.
     \end{split}
\end{equation}
\end{proof}

\subsection{The physical problem}
\label{subsec:physical}

We record here the well-posedness of the physical problem
\begin{equation}\label{eq:real_helmholtz}
  -c^2\Delta u - \omega^2 u = f \ \text{ in } \R^2,
  \qquad
  \lim_{r\to\infty} r^{1/2}\Bigl(c\,\frac{\p u}{\p r}-\ri\omega u\Bigr)=0,
\end{equation}
by \emph{outgoing solution} of which we mean a function
$u\in H^1_{\mathrm{loc}}(\R^2)$ satisfying the limit uniformly in the
angular variable and the equation in the weak sense
\begin{equation}\label{eq:weak_physical}
  \int_{\R^2}\na u\cdot\na\bar\varphi\,d\x
  -\omega^2\int_{\R^2}c^{-2}u\bar\varphi\,d\x
  =\int_{\R^2}c^{-2}f\bar\varphi\,d\x,
  \qquad\varphi\in C_c^\infty(\R^2).
\end{equation}
Dividing by $c^2$ is what makes the equation meaningful for a
coefficient that is merely bounded: the product $c^2\Delta u$ of an
$L^\infty$ function with a distribution is not defined, whereas
$c^{-2}u$ is. Formulation~\eqref{eq:weak_physical} is the one used
throughout, and it is the $\sigma_1=\sigma_2=0$ case of the PML weak
formulation~\eqref{eq:PML_extended_problem}.
Only~\ref{hyp:support} and~\ref{hyp:positive} are used: no regularity of
$c$ beyond boundedness, and no ray-geometric hypothesis. Throughout,
\begin{equation}\label{eq:Phi0}
  k=\frac{\omega}{c_\infty},
  \qquad
  \Phi_0(\x)=\frac{\ri}{4}H_0^{(1)}\bigl(k|\x|\bigr),\quad\x\ne0,
\end{equation}
so that $\Phi_0$ is the outgoing fundamental solution of $-\Delta-k^2$
in $\R^2$.

\begin{pro}\label{prop:physical}
Assume~\ref{hyp:support} and~\ref{hyp:positive}, and let $\omega>0$.
\begin{enumerate}
  \item\label{it:phys_uniq} \emph{(Uniqueness.)} If
        $v\in H^1_{\mathrm{loc}}(\R^2)$ satisfies $-c^2\Delta v-\omega^2v=0$
        in $\R^2$ together with the Sommerfeld condition
        in~\eqref{eq:real_helmholtz}, then $v\equiv0$.
  \item\label{it:phys_exist} \emph{(Existence.)} For every compactly
        supported $f\in L^2(\R^2)$ there is exactly one solution $u$
        of~\eqref{eq:real_helmholtz}. It satisfies
        $u\in H^2_{\mathrm{loc}}(\R^2)$, and it is real analytic on
        every open set on which $c\equiv c_\infty$ and $f\equiv0$.
\end{enumerate}
\end{pro}

\begin{proof}
\emph{(\ref{it:phys_uniq})} Since $\Delta v=-\omega^2c^{-2}v\in
L^2_{\mathrm{loc}}(\R^2)$ by~\ref{hyp:positive}, elliptic regularity for
the Laplacian gives $v\in H^2_{\mathrm{loc}}(\R^2)$. The potential
$\omega^2c^{-2}$ is real valued, so Green's identity on $B_\varrho$
yields
\begin{equation}
  \Im\int_{\p B_\varrho}\bar v\,\p_rv\,ds
  =\frac{1}{2\ri}\int_{B_\varrho}\bigl(\bar v\,\Delta v-v\,\overline{\Delta v}\bigr)
  =0
  \qquad\text{for every }\varrho>0 .
\end{equation}
Combined with the Sommerfeld condition this forces
$\int_{|\x|=\varrho}|v|^2\,ds\to0$ as $\varrho\to\infty$, and Rellich's
lemma gives $v\equiv0$ on $\{|\x|>R\}$. On $\R^2$ the function $v$
solves $\Delta v+\omega^2c^{-2}v=0$, an equation whose principal part is
the Laplacian and whose potential $\omega^2c^{-2}$ lies in
$L^\infty(\R^2)$; the unique continuation property for such operators
holds with no regularity assumption on the potential
\cite{Aronszajn1957}, so $v\equiv0$ on $\R^2$.

\emph{(\ref{it:phys_exist})} Choose $R_1\ge R$ with
$\supp f\subset\overline{B_{R_1}}$ and set
$V=\omega^2\bigl(c^{-2}-c_\infty^{-2}\bigr)$, which
by~\ref{hyp:support} and~\ref{hyp:positive} belongs to $L^\infty(\R^2)$
and is supported in $\overline{B_R}$. Equation~\eqref{eq:real_helmholtz}
may be rewritten as
\begin{equation}\label{eq:helm_pert}
  -\Delta u-k^2u=c^{-2}f+Vu \quad\text{in }\R^2,
\end{equation}
with the same radiation condition. If $u$ is a radiating solution
of~\eqref{eq:helm_pert}, then
$w=u-\Phi_0*\bigl(c^{-2}f+Vu\bigr)$ is an entire radiating solution of
the homogeneous Helmholtz equation, hence $w\equiv0$ on the complement
of a large ball by Rellich's lemma and, being real analytic on $\R^2$,
$w\equiv0$ everywhere. Conversely, the volume potential of an $L^2$ density with compact
support is radiating and solves~\eqref{eq:helm_pert}. Thus
$u$ solves~\eqref{eq:real_helmholtz} if and only if
\begin{equation}\label{eq:LS}
  u-\mathcal K u=F,
  \qquad
  \mathcal Ku=\Phi_0*(Vu),
  \qquad
  F=\Phi_0*\bigl(c^{-2}f\bigr),
\end{equation}
and, $V$ and $c^{-2}f$ being supported in $\overline{B_{R_1}}$, it
suffices to solve~\eqref{eq:LS} in $L^2(B_{R_1})$. The volume potential
maps $L^2(B_{R_1})$ boundedly into $H^2(B_{R_1})$ (see,
e.g.,~\cite[Ch.~8]{coltonkress}), so $\mathcal K$ is compact on
$L^2(B_{R_1})$ by Rellich--Kondrachov. Any element of
$\ker(I-\mathcal K)$ extends, through the right-hand side
of~\eqref{eq:LS}, to a radiating solution of $-c^2\Delta u-\omega^2u=0$,
which vanishes by part~(\ref{it:phys_uniq}); the Riesz--Fredholm theory
therefore makes $I-\mathcal K$ bijective on $L^2(B_{R_1})$ with bounded
inverse, and~\eqref{eq:LS} has exactly one solution. Finally
$\Delta u=-c^{-2}\bigl(\omega^2u+f\bigr)\in L^2_{\mathrm{loc}}(\R^2)$
gives $u\in H^2_{\mathrm{loc}}(\R^2)$, and on
$\R^2\setminus\overline{B_{R_1}}$ the function $u$ satisfies
$\Delta u+k^2u=0$, hence is real analytic there.
\end{proof}

\begin{rmk}\label{rmk:no_ntrap_needed}
Proposition~\ref{prop:physical} is the only property
of~\eqref{eq:real_helmholtz} used in
Sections~\ref{sec:allomega} and~\ref{sec:convergence}. In particular
neither the non-trapping condition~\ref{hyp:morawetz} of
Proposition~\ref{prop:non_trapping} nor any smoothness of $c$ enters
the main results; a piecewise constant speed, with the interfaces that
occur in layered media, is admissible throughout. What non-trapping
does buy is quantitative: it is what makes the Cauchy data
$\mathcal N_R(u)$ of Section~\ref{sec:convergence} grow at most
polynomially in $\omega$; see Remark~\ref{rmk:open}.
\end{rmk}

\subsection{The non-trapping condition}
\label{subsec:nontrapping}

\begin{pro}\label{prop:non_trapping}
Let $c\in C^{1,1}(\R^2)$ satisfy~\ref{hyp:support}, \ref{hyp:positive} and
\begin{enumerate}[label=(NT),ref=(NT)]
    \item\label{hyp:morawetz} There exists a point $\mathbf{x}_0\in\R^2$ such that
          \begin{equation}\label{eq:morawetz_cond}
              \nabla c(x,y)\cdot\bigl((x,y)-\mathbf{x}_0\bigr) \;<\; c(x,y)
              \qquad\forall\,(x,y)\in\R^2.
          \end{equation}
\end{enumerate}
Since $c\in C^{1,1}(\R^2)$ and, by~\ref{hyp:support}, $\na c\equiv0$ on
$\{|\x|\ge R\}$, the function
$(x,y)\mapsto1-c(x,y)^{-1}\nabla c(x,y)\cdot\bigl((x,y)-\mathbf x_0\bigr)$
is continuous and identically equal to $1$ outside the compact set
$\overline{B_R}$, so~\ref{hyp:morawetz} is equivalent to the existence of
a constant $\mu\in(0,1]$ with
\begin{equation}\label{eq:mu_def}
  \nabla c(x,y)\cdot\bigl((x,y)-\mathbf{x}_0\bigr)\;\le\;(1-\mu)\,c(x,y)
  \qquad\forall\,(x,y)\in\R^2 .
\end{equation}
Then $c$ is non-trapping: every bicharacteristic of the Hamiltonian
$H(x,y,\xi,\eta)=c(x,y)\sqrt{\xi^2+\eta^2}$ satisfies
$\sqrt{x(t)^2+y(t)^2}\to+\infty$ as $|t|\to+\infty$. If moreover
$c\in C^\infty(\R^2)$, then for every $\omega>0$ and every compactly
supported $f\in L^2(\R^2)$ the limit
\begin{equation}\label{eq:LAP}
  \lim_{\epsilon\to0^+}\bigl(-c^2\Delta-(\omega^2+\ri\epsilon)I\bigr)^{-1}f
\end{equation}
exists in $H^1_\text{loc}(\R^2)$ and coincides with the solution $u$
of~\eqref{eq:real_helmholtz} given by
Proposition~\ref{prop:physical}(\ref{it:phys_exist}).
\end{pro}

\begin{proof}
Let $\gamma(t)=(x(t),y(t),\xi(t),\eta(t))$ be a bicharacteristic with
energy $E=c(x(t),y(t))\sqrt{\xi(t)^2+\eta(t)^2}>0$. Define the
\emph{Morawetz multiplier}
\begin{equation}\label{eq:morawetz_mult}
    M(t) = \bigl((x(t),y(t))-\mathbf{x}_0\bigr)\cdot(\xi(t),\eta(t)).
\end{equation}
The Hamilton equations give
\begin{equation}\label{eq:dMdt}
    \dot{M}(t)
    = c(x,y)\sqrt{\xi^2+\eta^2}
      - \sqrt{\xi^2+\eta^2}\,
        \nabla c(x,y)\cdot\bigl((x,y)-\mathbf{x}_0\bigr).
\end{equation}
It is essential \emph{not} to estimate the two terms
of~\eqref{eq:dMdt} separately: the multiplier~\eqref{eq:morawetz_mult} is
the flat one, so the propagation speed enters only through the
combination in which the two terms appear. Factoring out
$\sqrt{\xi^2+\eta^2}$ and then using the energy identity
$c(x,y)\sqrt{\xi^2+\eta^2}=E$ once, we obtain the exact identity
\begin{equation}\label{eq:Mdot_identity}
\begin{split}
    \dot M(t)
    &= \sqrt{\xi^2+\eta^2}\,
      \Bigl[\,c(x,y)-\nabla c(x,y)\cdot\bigl((x,y)-\mathbf x_0\bigr)\Bigr]\\
    &= E\left[1-\frac{\nabla c(x,y)\cdot\bigl((x,y)-\mathbf x_0\bigr)}{c(x,y)}\right],
\end{split}
\end{equation}
in which the relevant quantity is the \emph{scale-invariant} ratio
$c^{-1}\nabla c\cdot(\x-\mathbf x_0)$ rather than
$\nabla c\cdot(\x-\mathbf x_0)$ alone. By~\ref{hyp:morawetz},
in the form~\eqref{eq:mu_def},
\begin{equation}\label{eq:Mdot_lower}
    \dot{M}(t) \;\geq\; \mu E \;>\;0.
\end{equation}
Integrating gives $M(t)\geq M(0)+\mu Et\to+\infty$ as $t\to+\infty$. On
the other hand, by Cauchy--Schwarz and, using~\ref{hyp:positive},
$\sqrt{\xi^2+\eta^2}=E/c(x,y)\leq E/c_0$,
\begin{equation}
    M(t) \leq \bigl|(x(t),y(t))-\mathbf{x}_0\bigr|\cdot\frac{E}{c_0},
\end{equation}
so
\begin{equation}
    \bigl|(x(t),y(t))-\mathbf{x}_0\bigr|
    \;\geq\; \frac{c_{\text{min}}}{E}\,M(t)
    \;\geq\; \frac{c_{\text{min}}}{E}\bigl(M(0)+\mu Et\bigr)
    \;\longrightarrow\; +\infty.
\end{equation}
The flow is complete: $|\dot\x|=c\le c_{\text{max}}$ and
$|\boldsymbol\xi|=E/c\in[E/c_{\text{max}},E/c_{\text{min}}]$, so no
trajectory escapes to infinity in finite time or reaches
$\boldsymbol\xi=0$, where $H$ fails to be differentiable. As
$t\to-\infty$ the same inequality $\dot M\ge\mu E$ gives
$M(t)\le M(0)+\mu Et\to-\infty$, whence
$|\x(t)-\x_0|\ge c_{\text{min}}|M(t)|/E\to+\infty$. Hence $c$ is
non-trapping.

Note that only $\na c$ has been used, and only through
Lipschitz continuity of $\na c$, which is what makes the Hamiltonian
flow well defined; the argument is therefore valid in $C^{1,1}(\R^2)$.

It remains to say what non-trapping buys, and here a distinction must be
drawn. The existence of the limit~\eqref{eq:LAP}, and with it of the
outgoing solution, does \emph{not} require it: that is already
Proposition~\ref{prop:physical}, proved under~\ref{hyp:support}
and~\ref{hyp:positive} alone. What non-trapping buys is quantitative. It
is the geometric hypothesis under which the cut-off resolvent obeys the
optimal high-frequency bound
\begin{equation}\label{eq:ntrap_resolvent}
  \bigl\|\chi\bigl(-c^2\Delta-(\omega^2+\ri0)\bigr)^{-1}\chi\bigr\|
  _{L^2(\R^2)\to L^2(\R^2)}\ \le\ \frac{C}{\omega},
  \qquad \omega\ge\omega_0,
\end{equation}
for every $\chi\in C_c^\infty(\R^2)$. Estimate~\eqref{eq:ntrap_resolvent}
is due to Burq~\cite{Burq2002}, who obtains resonance-free regions and
resolvent bounds in them for non-trapping geometries in any dimension;
the constant $C$ was later made explicit, in terms of the length of the
longest bicharacteristic meeting $\supp\chi$, by Galkowski, Spence and
Wunsch~\cite{GalkowskiSpenceWunsch2020}, who work on manifolds with
Euclidean ends --- the class to which $-c^2\Delta$ belongs
by~\ref{hyp:support}. Both statements assume a smooth speed.

Without any geometric hypothesis the cut-off resolvent remains finite
for each $\omega$, but the best general bound is exponential,
$\|\chi R(\omega)\chi\|\le e^{\alpha\omega}$, proved for arbitrary
obstacles by Burq~\cite{Burq1998} and refined away from the trapped set
by Cardoso and Vodev~\cite{CardosoVodev2002}. This is the precise sense
in which a trapping medium is admissible in
Theorem~\ref{thm:convergence} and yet useless in the high-frequency
limit; see Remark~\ref{rmk:open}.
\end{proof}

\begin{rmk}\label{rmk:mollify}
The two halves of Proposition~\ref{prop:non_trapping} are stated under
different regularity because the second is a citation, not a technique:
the Morawetz computation needs $c\in C^{1,1}$, while the limiting
absorption principle is quoted from a smooth theory. Condition
\ref{hyp:morawetz} is stable under mollification within $C^{1,1}$: if
$\eta_\epsilon$ is a standard mollifier supported in $B_\epsilon$ and
$c_\epsilon=c*\eta_\epsilon$, then, writing
$\x-\x_0=\bigl((\x-\y)-\x_0\bigr)+\y$ and using~\eqref{eq:mu_def}
pointwise,
\begin{equation}\label{eq:mollify}
\begin{split}
  \na c_\epsilon(\x)\cdot(\x-\x_0)
  &=\int \na c(\x-\y)\cdot\bigl((\x-\y)-\x_0\bigr)\eta_\epsilon(\y)\,d\y\\
  &\qquad+\int \na c(\x-\y)\cdot\y\,\eta_\epsilon(\y)\,d\y\\
  &\le(1-\mu)\,c_\epsilon(\x)+\epsilon\,\|\na c\|_{L^\infty(\R^2)} ,
\end{split}
\end{equation}
so that, since $c_\epsilon\ge c_{\text{min}}$, the
speed $c_\epsilon$ satisfies~\eqref{eq:mu_def} with $\mu$ replaced by
$\mu/2$ as soon as
$\epsilon\le\mu\,c_{\text{min}}/\bigl(2\|\na c\|_{L^\infty}\bigr)$; the
constants $c_{\text{min}},c_{\text{max}}$ are preserved and $R$ is
replaced by $R+\epsilon$. What
this does \emph{not} yet supply is the interchange of the two limits
$\epsilon\to0$ and the absorption limit in~\eqref{eq:LAP}, which would
require a resolvent bound uniform in $\epsilon$. Such a bound does not
follow from~\cite{Burq2002,GalkowskiSpenceWunsch2020} as stated, since
their constants are not exhibited in terms of
$\mu,c_{\text{min}},c_{\text{max}},R$ alone; a Morawetz--Rellich identity carried out at the level of the
partial differential equation rather than of the bicharacteristics would
supply it, and would at the same time make the estimate of
Section~\ref{sec:convergence} explicit in $\omega$. We do not pursue
this here.

The computation~\eqref{eq:mollify} uses only $\na c\in L^\infty$, so it
is available in $W^{1,\infty}(\R^2)$; what needs $C^{1,1}$ is the
Hamiltonian flow, not the mollification. At $L^\infty$ regularity,
however, the device fails, and for a structural reason. If $c$ merely lies in $L^\infty$ and jumps across an interface
then $|\na c_\epsilon|\sim\epsilon^{-1}$ there and the left-hand side
of~\eqref{eq:mollify} diverges. The failure is not an artefact of the
method: a piecewise constant speed with a low-velocity inclusion
$c\equiv c_-<c_\infty$ on $B_R$ traps, by total internal reflection,
every ray whose impact parameter exceeds $R\,c_-/c_\infty$, a set of
positive measure. No non-trapping statement can survive at that
regularity, which is precisely why the existence theory of
Proposition~\ref{prop:physical} is formulated without one.
\end{rmk}

\begin{rmk}\label{rmk:sharpness}
Condition~\ref{hyp:morawetz} does \emph{not} say that $c$ is
non-increasing along rays emanating from $\mathbf x_0$. It allows $c$ to
increase, provided it does so more slowly than linearly relative to its
own size: what is bounded is the logarithmic derivative
$\p_\varrho\log c$ along the ray, where $\varrho=|\x-\mathbf x_0|$, and
the condition reads $\varrho\,\p_\varrho\log c<1$. This is exactly the
sharp threshold. Indeed, for a radial profile $c=c(\varrho)$ centred at
$\mathbf x_0$, \eqref{eq:morawetz_cond} becomes
$\varrho\,c'(\varrho)<c(\varrho)$, i.e.
\begin{equation}\label{eq:herglotz}
  \frac{d}{d\varrho}\left(\frac{\varrho}{c(\varrho)}\right)>0,
\end{equation}
which is the classical Herglotz condition of geometrical seismology: the
Clairaut invariant of a ray in the metric $c^{-2}\delta$ is
$\varrho\sin\alpha/c(\varrho)$, so \eqref{eq:herglotz} is precisely the
statement that every ray has at most one turning point and escapes.
Its failure at some radius $\varrho_0$, i.e.
$c'(\varrho_0)=c(\varrho_0)/\varrho_0$, is the circular-orbit condition,
and a trapped bicharacteristic then exists. Thus~\ref{hyp:morawetz}
cannot be weakened within this class.

The stronger requirement $\nabla c\cdot(\x-\mathbf x_0)\le0$, i.e.
$\mu=1$ in~\eqref{eq:mu_def}, would rule out \emph{every} local minimum
of $c$, and in particular every low-velocity inclusion, however weak;
these are the media of greatest interest in applications and only become
trapping once the well is deep or steep enough
to violate~\eqref{eq:herglotz}. Note also that~\ref{hyp:support} alone
does not imply~\ref{hyp:morawetz}: a smooth $c$ equal to $c_\infty$
outside a compact set may still violate~\eqref{eq:herglotz}.

It may help to relate~\ref{hyp:morawetz} to the hypotheses under which
frequency-explicit bounds are available for rough media. Chaumont-Frelet
and Spence~\cite{ChaumontFreletSpence2023} obtain such bounds for
coefficients that are merely bounded, under a monotonicity condition
along one spatial direction. Written as
$\p_\varrho\bigl(\varrho/c\bigr)>0$, condition~\ref{hyp:morawetz} is a
monotonicity condition of the same kind, along rays issuing from
$\x_0$; the difference is that we require it of $\na c$ pointwise,
which forces $c\in C^{1,1}$, whereas their formulation survives jumps of
the right sign and is therefore available at $L^\infty$ regularity.
Making~\ref{hyp:morawetz} itself meaningful for a discontinuous speed,
presumably in a distributional or monotone-rearrangement sense, and
recovering a quantitative resolvent bound from it, is the natural way
to attack the frequency-uniformity question left open in
Remark~\ref{rmk:open}.
\end{rmk}

\section{The truncated problem}
\label{sec:truncated}

We now turn to the truncated problem~\eqref{eq:PML_truncated_problem},
posed on the bounded rectangle $\Omega_\delta$ with homogeneous
Dirichlet boundary condition. Let $T_\delta$ denote the unbounded
operator on $L^2(\Omega_\delta)$ with maximal domain
\begin{equation}
  \mathcal{D}(T_\delta) = \bigl\{u \in H^1_0(\Omega_\delta) : \widetilde\Delta u \in L^2(\Omega_\delta)\bigr\},
  \qquad T_\delta u = -c^2\widetilde\Delta u.
\end{equation}

For a closed operator $T$ we write $\Sigma_w(T)$ for the set of
$\lambda\in\C$ such that $T-\lambda I$ fails to be a Fredholm operator
of index zero.

\begin{pro}\label{prop:compact_resolvent}
Let $w_0=-e^{-\ri\theta}$ be as in Lemma~\ref{lm:well-possed}. Then
$w_0\in\rho(T_\delta)$, and $(T_\delta-w_0I)^{-1}$ is a compact
operator on $L^2(\Omega_\delta)$.
\end{pro}

\begin{proof}
The pointwise bounds of Lemma~\ref{lemma:bounds1} hold for every
$(x,y)\in\R^2$, hence in particular on $\Omega_\delta$; the proof of
Lemma~\ref{lm:well-possed} used no property of $\R^2$ beyond these
pointwise bounds, so $\mathcal{A}_{\Omega_\delta,w_0}$ is continuous
and coercive on $H^1_0(\Omega_\delta)$ with the same constants
$C_1,C_2$.

For $f\in L^2(\Omega_\delta)$, Lemma~\ref{lemma:bounds1}(2) gives
$\|(Jf,\cdot)\|_{H^{-1}(\Omega_\delta)}\le
c_\text{min}^{-2}(\sigma_M^2+1)\|f\|_{L^2(\Omega_\delta)}$, so by
Lax--Milgram there is a unique $u\in H^1_0(\Omega_\delta)$ with
\begin{equation}
  \|u\|_{H^1(\Omega_\delta)} \le
  \frac{c_\text{min}^{-2}(\sigma_M^2+1)}{C_2}\,\|f\|_{L^2(\Omega_\delta)}.
\end{equation}
This defines a bounded operator $R_{w_0}:L^2(\Omega_\delta)\to
H^1_0(\Omega_\delta)$, $f\mapsto u=(T_\delta-w_0I)^{-1}f$; in
particular $w_0\in\rho(T_\delta)$.

The embedding $\iota: H^1_0(\Omega_\delta)\hookrightarrow
L^2(\Omega_\delta)$ is compact by the Rellich--Kondrachov theorem:
extending by zero identifies $H^1_0(\Omega_\delta)$ isometrically
with a closed subspace of $H^1(B)$ for any ball
$B\supset\overline{\Omega_\delta}$, where Rellich--Kondrachov applies
without any regularity requirement on $\partial\Omega_\delta$. Hence
\begin{equation}
  (T_\delta-w_0I)^{-1} = \iota\circ R_{w_0} :
  L^2(\Omega_\delta) \to L^2(\Omega_\delta)
\end{equation}
is the composition of a bounded operator with a compact one, hence
compact.
\end{proof}

\begin{cor}\label{cor:discrete_spectrum}
$\Sigma_w(T_\delta)=\emptyset$. The spectrum $\sigma(T_\delta)$
consists of isolated eigenvalues of finite algebraic multiplicity,
with no accumulation point in $\C$; in particular $\sigma_p(T_\delta)$
is at most countable.
\end{cor}

\begin{proof}
Write $K=(T_\delta-w_0I)^{-1}$, compact by
Proposition~\ref{prop:compact_resolvent}. For $u\in\mathcal D(T_\delta)$,
\begin{equation}
  K(T_\delta-\lambda I)u = K(T_\delta - w_0I)u - (\lambda-w_0)Ku
  = u - (\lambda-w_0)Ku, \qquad \lambda\in\C,
\end{equation}
so $u\ne0$ solves $(T_\delta-\lambda I)u=0$ if and only if
$Ku=(\lambda-w_0)^{-1}u$ (for $\lambda\ne w_0$; $\lambda=w_0$ is
already excluded since $w_0\in\rho(T_\delta)$). By the classical
Riesz--Schauder theory of compact operators, the nonzero eigenvalues
of $K$ form a set with no accumulation point other than possibly $0$,
each of finite multiplicity. Consequently the corresponding values
$\lambda=w_0+\mu^{-1}$, as $\mu$ ranges over the nonzero eigenvalues
of $K$, are isolated, of finite algebraic multiplicity, with no
finite accumulation point, and constitute exactly $\sigma(T_\delta)$.
In particular $\Sigma_w(T_\delta)=\emptyset$.
\end{proof}

\begin{pro}\label{prop:truncated_wellposed}
Let $\omega>0$ be such that $\omega^2\notin\sigma_p(T_\delta)$. Then
for every $f\in L^2(\Omega_\delta)$ there exists a unique
$u\in H^1_0(\Omega_\delta)$ solving~\eqref{eq:PML_truncated_problem}
with $w=\omega^2$, depending continuously on $f$. This holds for
every $\omega>0$ outside an at most countable set.
\end{pro}

\begin{proof}
By Corollary~\ref{cor:discrete_spectrum}, $T_\delta-\omega^2I$ is
Fredholm of index zero for every $\omega^2\in\C$: writing
$T_\delta-\omega^2I = (T_\delta-w_0I)\bigl(I-(\omega^2-w_0)K\bigr)$
formally on $\mathcal D(T_\delta)$, the factor
$I-(\omega^2-w_0)K$ is Fredholm of index $0$ on $L^2(\Omega_\delta)$
for every scalar $(\omega^2-w_0)$ by the Riesz--Schauder alternative
for the compact operator $K$, and $(T_\delta-w_0I)$ is bijective;
composing preserves Fredholm index $0$. A Fredholm operator of index
zero is bijective iff injective; the hypothesis
$\omega^2\notin\sigma_p(T_\delta)$ is exactly
$\ker(T_\delta-\omega^2I)=\{0\}$, giving bijectivity with bounded
inverse. The last claim follows since $\sigma_p(T_\delta)$ is at most
countable.
\end{proof}

\section{Well-posedness for every frequency}
\label{sec:allomega}

The results of Section~\ref{sec:truncated} leave two blemishes: the
truncated problem is solvable only for $\omega$ outside an at most
countable set, and no control is offered on the size of the inverse.
This section removes both. The mechanism is a localisation lemma which
shows that, as far as the essential spectrum is concerned, the variable
speed is invisible: the operator may be replaced by its homogeneous
counterpart, for which the spectrum was computed by Kim and
Pasciak~\cite{KimPasciak2010a}.

From here on we assume, in addition
to~\ref{hyp:support}--\ref{hyp:positive} and to
$0<\gamma_a\le\gamma_b=\sigma_M<1$:

\begin{enumerate}[label=(H\arabic*),ref=(H\arabic*),start=3]
  \item\label{hyp:geom} $R<\min\{a,b\}$ and $\supp f\subset\overline{B_R}$,
        where $B_R=\{\x\in\R^2:|\x|<R\}$;
  \item\label{hyp:saturate} the layer is at least twice as wide as the
        transition, $\delta_a\ge2\delta^{\mathrm t}_a$ and $\delta_b\ge2\delta^{\mathrm t}_b$; hence
        $\sigma_1\equiv\gamma_a$ on $\{|x|\ge a+\tfrac12\delta_a\}$ and
        $\sigma_2\equiv\gamma_b$ on $\{|y|\ge b+\tfrac12\delta_b\}$, so
        that the outer half of the layer is a constant-coefficient
        region.
\end{enumerate}

\begin{rmk}\label{rmk:H5_free}
Hypothesis~\ref{hyp:saturate} costs nothing: the transition widths
$\delta^{\mathrm t}_a,\delta^{\mathrm t}_b$ are fixed data of the profile, while $\delta_a,\delta_b$
are truncation distances, and~\ref{hyp:saturate} merely asks that
truncation occur beyond twice the transition width. It is essential,
however, that $\delta^{\mathrm t}_a,\delta^{\mathrm t}_b$ be independent of $\delta_a,\delta_b$, as
stipulated in Section~\ref{sec:preliminaries}: the operator
$-c^2\widetilde\Delta$ on $\R^2$ must not change with $\delta$, because
Proposition~\ref{prop:uniform} compares the truncated problems for a
whole sequence of layer widths against \emph{one} fixed operator on the
plane. Tying the transition width to $\delta$ --- a tempting
simplification --- would make that comparison vacuous. This is also the
arrangement adopted
in~\cite{KimPasciak2010a,KimPasciak2010b}. Hypothesis~\ref{hyp:saturate}
is used wherever a reflection across $\p\Omega_\delta$ must leave the
coefficients invariant --- in
Proposition~\ref{prop:uniform} and in the $H^2$ regularity of
$\mathcal D(T_\delta)$ established there --- and, in a purely
quantitative way, in Corollary~\ref{cor:exp_rate}.
\end{rmk}

\subsection{Further notation}

Recall $k=\omega/c_\infty$ from~\eqref{eq:Phi0} and put, for $j=1,2$,
\begin{equation}\label{eq:def_tau}
  \tau_1(x)=\int_0^x\sigma_1(s)\,ds = x\tilde\sigma_1(x),
  \qquad
  \tau_2(y)=\int_0^y\sigma_2(s)\,ds = y\tilde\sigma_2(y),
\end{equation}
so that, with $z=\ri$, the stretching map of
Section~\ref{sec:preliminaries} reads
\begin{equation}\label{eq:T_as_tau}
  T(\x)=\x+\ri\,\boldsymbol\tau(\x),
  \qquad \boldsymbol\tau(\x)=(\tau_1(x),\tau_2(y)),
  \qquad \x=(x,y).
\end{equation}
For $s>0$ put
\begin{equation}\label{eq:def_Sigma}
  \Sigma_1(s)=\int_a^{a+s\delta_a}\sigma_1(t)\,dt>0,
  \qquad
  \Sigma_2(s)=\int_b^{b+s\delta_b}\sigma_2(t)\,dt>0,
\end{equation}
the total absorption accumulated across a fraction $s$ of the layer in
each direction, and
\begin{equation}
  \Omega^{(s)}=(-a-s\delta_a,a+s\delta_a)\times(-b-s\delta_b,b+s\delta_b),
\end{equation}
so that $\Omega^{(0^+)}=\Omega$ and $\Omega^{(1)}=\Omega_\delta$. The
slightly enlarged rectangle $\Omega^{(2)}$ will be used as a collar in
which interior estimates up to $\p\Omega_\delta$ can be performed. Let
$\rho:\C^2\to\C$ be the complexified distance
\begin{equation}\label{eq:def_rho}
  \rho(\mathbf z)=\left(z_1^2+z_2^2\right)^{1/2},
  \qquad \mathbf z=(z_1,z_2)\in\C^2,
\end{equation}
where the principal branch of the square root on $\C\setminus(-\infty,0]$
is used, and let
\begin{equation}\label{eq:def_Phi}
  \Phi(\mathbf z)=\frac{\ri}{4}H_0^{(1)}\bigl(k\rho(\mathbf z)\bigr)
\end{equation}
be the analytic continuation of the outgoing fundamental
solution $\Phi_0$ of~\eqref{eq:Phi0}, so that $\Phi(\x-\y)=\Phi_0(\x-\y)$
for real arguments. Finally, set
\begin{equation}\label{eq:def_Lambda}
  \Lambda_\delta=\sqrt{(a+2\delta_a)^2+(b+2\delta_b)^2}+R
                 +2\sigma_M\sqrt{\delta_a^2+\delta_b^2},
\end{equation}
and, for $s\in(0,1)$,
\begin{equation}\label{eq:def_Theta}
  \Theta(s)=\frac{1}{\Lambda_\delta}
  \min\Bigl\{\Sigma_1(s)\bigl(a+s\delta_a-R\bigr),\;
             \Sigma_2(s)\bigl(b+s\delta_b-R\bigr)\Bigr\}>0 .
\end{equation}

\subsection{Localisation of the essential spectrum}

\begin{lemma}\label{lem:reduction}
Let $-c^2\widetilde\Delta$ and $-c_\infty^2\widetilde\Delta$ be regarded
as unbounded operators on $L^2(\R^2)$ with domain $H^2(\R^2)$. Then, for
every $\lambda\in\C$, $-c^2\widetilde\Delta-\lambda I$ is Fredholm of
index zero if and only if $-c_\infty^2\widetilde\Delta-\lambda I$ is.
Consequently
\begin{equation}\label{eq:reduction}
  \Sigma_w\bigl(-c^2\widetilde\Delta\bigr)
  =\Sigma_w\bigl(-c_\infty^2\widetilde\Delta\bigr).
\end{equation}
\end{lemma}

\begin{proof}
Since $c$ is bounded above and below by positive constants, the
multiplication operators $M_{c^2}u=c^2u$ and $M_{c_\infty^2}$ are
automorphisms of $L^2(\R^2)$. Factor
\begin{equation}\label{eq:factorisation}
  -c^2\widetilde\Delta-\lambda I
  =M_{c^2}\bigl(-\widetilde\Delta-\lambda c^{-2}\bigr),
  \qquad
  -c_\infty^2\widetilde\Delta-\lambda I
  =M_{c_\infty^2}\bigl(-\widetilde\Delta-\lambda c_\infty^{-2}\bigr),
\end{equation}
both understood as operators $H^2(\R^2)\to L^2(\R^2)$. Their difference
inside the brackets is the multiplication operator
\begin{equation}
  \bigl(-\widetilde\Delta-\lambda c^{-2}\bigr)
  -\bigl(-\widetilde\Delta-\lambda c_\infty^{-2}\bigr)
  =\lambda\bigl(c_\infty^{-2}-c^{-2}\bigr),
\end{equation}
and by~\ref{hyp:support} and~\ref{hyp:positive} the function
$c_\infty^{-2}-c^{-2}$ is bounded and supported in $\overline{B_R}$; no
regularity of $c$ is needed, only the two-sided bound. The map
$u\mapsto\lambda(c_\infty^{-2}-c^{-2})u$ is therefore compact from
$H^2(\R^2)$ into $L^2(\R^2)$: a bounded sequence in $H^2(\R^2)$ is
bounded in $H^2(B)$ for a ball $B\supset\overline{B_R}$, and
$H^2(B)\hookrightarrow L^2(B)$ compactly by
Rellich--Kondrachov. Fredholmness and the index are stable under
compact perturbations, and composition with an automorphism of
$L^2(\R^2)$ changes neither. The equivalence follows
from~\eqref{eq:factorisation}, and~\eqref{eq:reduction} is a
restatement of it.
\end{proof}

\begin{rmk}\label{rmk:why_factorise}
The factorisation~\eqref{eq:factorisation} is the whole point, and it is
worth isolating why the direct comparison fails. Subtracting the two
operators as they stand gives
$\bigl(c_\infty^2-c^2\bigr)\widetilde\Delta$, a \emph{second order}
operator whose coefficients are supported in $\overline{B_R}$; such an
operator is not relatively compact with respect to
$\widetilde\Delta$, because $H^2(B)\to L^2(B)$, $u\mapsto\chi
D^2u$, is bounded but not compact. Factoring out $c^2$ moves the
compactly supported factor from the principal part to the zeroth order
term, where Rellich--Kondrachov applies. The variable speed thus drops
out of the spectral analysis entirely, although it certainly does not
drop out of the analysis of Section~\ref{sec:convergence}, where it
enters through the Cauchy data of the physical solution on $\p B_R$.
\end{rmk}

Let
\begin{equation}\label{eq:def_sector}
  S=\Bigl\{z\in\C\setminus\{0\}:\ -2\Arg(1+\ri\gamma_b)\le\Arg z
  \le-2\Arg(1+\ri\gamma_a)\Bigr\}\cup\{0\},
\end{equation}
the closed sector spanned by the two rays
$\{\xi^2/(1+\ri\gamma_a)^2:\xi\in\R\}$ and
$\{\eta^2/(1+\ri\gamma_b)^2:\eta\in\R\}$; it degenerates to a single ray
when $\gamma_a=\gamma_b$.

\begin{pro}\label{prop:sector}
$\Sigma_w\bigl(-c^2\widetilde\Delta\bigr)=S$. In particular
\begin{equation}
  S\cap[0,\infty)=\{0\},
\end{equation}
so that for every $\omega>0$ the operator
$-c^2\widetilde\Delta-\omega^2I:H^2(\R^2)\to L^2(\R^2)$ is Fredholm of
index zero.
\end{pro}

\begin{proof}
For the homogeneous operator, $-c_\infty^2\widetilde\Delta
=c_\infty^2\bigl(-\widetilde\Delta\bigr)$, and
\cite[Thm.~4.5]{KimPasciak2010a} gives
$\sigma\bigl(-\widetilde\Delta\bigr)
=\sigma_{\mathrm{ess}}\bigl(-\widetilde\Delta\bigr)=S$ on $L^2(\R^2)$
with domain $H^2(\R^2)$. As $S$ is a cone, $c_\infty^2S=S$, so
$\sigma\bigl(-c_\infty^2\widetilde\Delta\bigr)=S$.

We must pass from $\sigma$ to $\Sigma_w$. If $\lambda\notin S$ then
$-c_\infty^2\widetilde\Delta-\lambda I$ is boundedly invertible, hence
Fredholm of index zero, so
$\Sigma_w\bigl(-c_\infty^2\widetilde\Delta\bigr)\subseteq S$.
Conversely, for every $\lambda\in S$ the explicit construction
of~\cite[Rem.~4.7]{KimPasciak2010a} --- cutting off the exponentials
$e^{\ri[\xi x/(1+\ri\gamma_a)+\eta y/(1+\ri\gamma_b)]}$, whose symbol value
$\xi^2/(1+\ri\gamma_a)^2+\eta^2/(1+\ri\gamma_b)^2$ sweeps all of $S$ ---
produces a normalised sequence $u_n$ with $u_n\rightharpoonup0$ and
$\bigl(-\widetilde\Delta-\lambda\bigr)u_n\to0$ in $L^2(\R^2)$. (The
cut-offs must escape to infinity inside one of the four corner regions
$\{|x|\ge a+\delta^{\mathrm t}_a\}\cap\{|y|\ge b+\delta^{\mathrm t}_b\}$, where
both profiles are saturated and the exponential is an exact solution;
this is where the Cartesian geometry enters.) An operator admitting such
a singular sequence is not semi-Fredholm, so
$S\subseteq\Sigma_w\bigl(-\widetilde\Delta\bigr)$, and multiplying by
$c_\infty^2$ and using $c_\infty^2S=S$ gives
$S\subseteq\Sigma_w\bigl(-c_\infty^2\widetilde\Delta\bigr)$.
Lemma~\ref{lem:reduction} transfers the resulting equality to
$-c^2\widetilde\Delta$.

Finally, $0<\gamma_a\le\gamma_b<1$ gives
$-2\Arg(1+\ri\gamma_j)\in(-\pi/2,0)$, so every nonzero point of $S$ has
strictly negative argument and $\omega^2\notin S$ for $\omega>0$.
\end{proof}

\begin{rmk}\label{rmk:limit_symbol}
It is worth recording why $S$ cannot be read off from the pointwise
behaviour of the symbol at infinity. Along the $x$-axis one has
$\sigma_2(y)=0$ for $|y|\le b$, so the naive limit symbol as
$|x|\to\infty$ with $y$ bounded would be
$c_\infty^2\bigl(\xi^2/(1+\ri\gamma_a)^2+\eta^2\bigr)$, whose range meets
$(0,\infty)$ and would wrongly suggest that real frequencies belong to
the essential spectrum. What rescues the situation is that the
directional limit is not a multiplication operator but the tensor sum
$A\otimes I+I\otimes B$, in which $B$ is the \emph{one-dimensional} PML
operator in $y$; the spectrum of $B$ is the ray
$\{\eta^2/(1+\ri\gamma_b)^2\}$ and not $[0,\infty)$, because $\sigma_2$ is
eventually equal to $\gamma_b$ and the one-dimensional operator has no
eigenvalues~\cite[Prop.~3.2]{KimPasciak2010a}. The tensor structure,
not the pointwise symbol, is what produces $S$.
\end{rmk}

\begin{rmk}\label{rmk:halla}
That $\Sigma_w$ should sit in the closed lower half plane, meeting
$[0,\infty)$ only at the origin, is a feature of the configuration and
not a general fact about absorbing layers. Halla, Kachanovska and
Wess~\cite{HallaKachanovskaWess2025} show that the radial PML for the
\emph{anisotropic} wave equation has non-empty essential spectrum in the
right half plane, and that this is the mechanism behind
discretisation-dependent instabilities observed in practice; the
threshold in their analysis is governed by the anisotropy ratio
$(\lambda_{\max}+\lambda_{\min})/2\sqrt{\lambda_{\max}\lambda_{\min}}$,
which degenerates to $1$ in the isotropic case. Here the medium is
isotropic and the stretching, although separable and therefore
directionally inhomogeneous, acts on a scalar Helmholtz operator; the
sector $S$ of~\eqref{eq:def_sector} is computed explicitly
in~\cite[Thm.~4.5]{KimPasciak2010a} and its two boundary rays have
argument $-2\Arg(1+\ri\gamma_j)\in(-\pi/2,0)$. By
Lemma~\ref{lem:reduction} a variable speed does not move it.
\end{rmk}

\subsection{Injectivity at real frequencies}

Fredholmness of index zero reduces invertibility to injectivity, which we
now establish for every $\omega>0$. The argument is that
of~\cite[Prop.~4.5]{KimPasciak2010b}: a nontrivial kernel element would,
through the integral representation, produce a nontrivial outgoing
solution of the physical problem.

\begin{pro}\label{prop:injective}
Assume~\ref{hyp:support}, \ref{hyp:positive} and~\ref{hyp:geom}, and
let $\omega>0$. If
$u\in H^2(\R^2)$ satisfies $-c^2\widetilde\Delta u=\omega^2u$, then
$u=0$. Consequently $\omega^2\in\rho\bigl(-c^2\widetilde\Delta\bigr)$
for every $\omega>0$.
\end{pro}

\begin{proof}
Fix $R'$ with $R<R'<\min\{a,b\}$, which is possible by~\ref{hyp:geom}.
On $\R^2\setminus\overline{B_R}$ we have $c\equiv c_\infty$, so $u$
satisfies there the homogeneous stretched equation
$\bigl(\widetilde\Delta+k^2\bigr)u=0$, $k=\omega/c_\infty$, with
$u\in H^1$; this is precisely the situation
of~\cite[Thm.~4.3]{KimPasciak2010b}, applied with $B_R$ in the role of
the scatterer and $B_{R'}$ in the role of the intermediate domain.
That theorem is a purely exterior statement: it uses only that
$u\in H^1$ and $(\widetilde\Delta+k^2)u=0$ on the complement of the
scatterer, not the equation inside it, so the variable speed on $B_R$ is
immaterial. It gives, for $\x\in\R^2\setminus\overline{B_{R'}}$, with $\n$
the unit normal on $\p B_{R'}$ pointing away from the origin,
\begin{equation}\label{eq:rep_infinite}
  u(\x)=\int_{\p B_{R'}}
  \left[u(\y)\,\frac{\p\Phi}{\p\n(\y)}\bigl(T(\x)-\y\bigr)
        -\frac{\p u}{\p\n}(\y)\,\Phi\bigl(T(\x)-\y\bigr)\right]ds(\y),
\end{equation}

Define $v$ on $\R^2$ by $v=u$ on $\overline{B_{R'}}$ and, for
$|\x|>R'$, by the right-hand side of~\eqref{eq:rep_infinite} with
$T(\x)$ replaced by $\x$, that is, with the unstretched fundamental
solution $\Phi_0(\x-\y)=\tfrac{\ri}{4}H_0^{(1)}(k|\x-\y|)$. Since
$\sigma_1$ and $\sigma_2$ vanish on $\overline\Omega$, we have
$T=\mathrm{id}$ there, and $\overline{B_{R'}}$ is compactly contained in
$\Omega$ by~\ref{hyp:geom}; hence $\Phi(T(\x)-\y)=\Phi_0(\x-\y)$ for
$\x$ in the annulus $R'<|\x|<\min\{a,b\}$, so the two definitions of
$v$ coincide there and $v$ is smooth across $\p B_{R'}$. Moreover $v$
is, by construction, a radiating solution of
$\Delta v+k^2v=0$ for $|\x|>R'$, and satisfies
$-c^2\Delta v-\omega^2v=0$ on $B_{R'}$ because $u$ does and $T=\mathrm{id}$
there. Hence $v$ solves
\begin{equation}
  -c^2\Delta v-\omega^2v=0 \ \text{ in }\R^2,
  \qquad
  \lim_{r\to\infty}r^{1/2}\Bigl(c\,\frac{\p v}{\p r}-\ri\omega v\Bigr)=0 .
\end{equation}
By Proposition~\ref{prop:physical}(\ref{it:phys_uniq}), $v\equiv0$. In
particular $u$ and $\p u/\p\n$ vanish on $\p B_{R'}$, and
then~\eqref{eq:rep_infinite} gives $u\equiv0$ on
$\R^2\setminus\overline{B_{R'}}$; together with $u=v=0$ on $B_{R'}$ this
yields $u=0$.

By Proposition~\ref{prop:sector}, $-c^2\widetilde\Delta-\omega^2I$ is
Fredholm of index zero; being injective, it is bijective with bounded
inverse, i.e.\ $\omega^2\in\rho\bigl(-c^2\widetilde\Delta\bigr)$.
\end{proof}

\begin{rmk}
The only input from the physical problem is the uniqueness statement
Proposition~\ref{prop:physical}(\ref{it:phys_uniq}), which rests on
Rellich's lemma and on unique continuation for a Schr\"odinger operator
with bounded potential. Neither a non-trapping hypothesis nor any
regularity of $c$ is involved.
\end{rmk}

\subsection{Stability of the truncated problem, uniformly in the layer width}
\label{subsec:uniform}

We now transfer the infinite-domain result to $\Omega_\delta$. The
argument is the reflection argument
of~\cite[Thm.~4.8]{KimPasciak2010a}; we indicate the modifications
required by the variable speed, all of which are innocuous because
$c\equiv c_\infty$ near $\p\Omega_\delta$.

\begin{pro}\label{prop:uniform}
Assume~\ref{hyp:support}, \ref{hyp:positive}, \ref{hyp:geom}
and~\ref{hyp:saturate}, and let $\omega>0$. Then there exist
$\delta_0=\delta_0(\omega)>0$ and $C_\star=C_\star(\omega)$ such that,
whenever $\delta_a,\delta_b\ge\delta_0$,
\begin{equation}\label{eq:uniform_infsup}
  \|u\|_{H^1(\Omega_\delta)}\ \le\ C_\star\!\!
  \sup_{0\ne\varphi\in H^1_0(\Omega_\delta)}
  \frac{\bigl|\mathcal A_{\Omega_\delta,\omega^2}(u,\varphi)\bigr|}
       {\|\varphi\|_{H^1(\Omega_\delta)}}
  \qquad\text{for all }u\in H^1_0(\Omega_\delta),
\end{equation}
and the same holds for the adjoint form. In particular
$\omega^2\notin\sigma_p(T_\delta)$, and the constant $C_\delta(\omega)$
of Lemma~\ref{lem:infsup} satisfies $C_\delta(\omega)\le C_\star(\omega)$
for all such $\delta$.
\end{pro}

\begin{proof}
Throughout, $-c^2\widetilde\Delta$ denotes \emph{one fixed} operator on
$L^2(\R^2)$: by the convention of Section~\ref{sec:preliminaries} the
profiles $\sigma_1,\sigma_2$ are determined by the fixed transition
widths $\delta^{\mathrm t}_a,\delta^{\mathrm t}_b$ and do not change when $\delta_a,\delta_b$ vary.
By Proposition~\ref{prop:injective},
$\omega^2\in\rho\bigl(-c^2\widetilde\Delta\bigr)$, so
\begin{equation}\label{eq:plane_resolvent}
  \|w\|_{L^2(\R^2)}\le C_\infty
  \bigl\|(-c^2\widetilde\Delta-\omega^2)w\bigr\|_{L^2(\R^2)}
  \qquad\text{for all }w\in H^2(\R^2),
\end{equation}
with $C_\infty=C_\infty(\omega)$ independent of everything else.

By~\cite[Lem.~2.5, Cor.~2.7 and Rem.~2.8]{KimPasciak2010a}, whose proofs use only
the coercivity of $\mathcal A_{\cdot,w_0}$ established in
Lemma~\ref{lm:well-possed} --- the role of $d_1d_2$ there being played
here by $J=c^{-2}d_1d_2$, which is bounded above and away from zero by
Lemma~\ref{lemma:bounds1}(2) --- it suffices to prove the uniform
resolvent bound
\begin{equation}\label{eq:uniform_resolvent}
  \|v\|_{L^2(\Omega_\delta)}\le C\,
  \bigl\|(-c^2\widetilde\Delta-\omega^2)v\bigr\|_{L^2(\Omega_\delta)},
  \qquad v\in\mathcal D(T_\delta),
\end{equation}
for $\delta_a,\delta_b\ge\delta_0$, with $C$ independent of $\delta$.
Note that~\eqref{eq:uniform_resolvent} yields not only the bound but
also the invertibility required in~\cite[Rem.~2.8]{KimPasciak2010a}:
by Corollary~\ref{cor:discrete_spectrum} the operator
$T_\delta-\omega^2I$ is Fredholm of index zero, and
\eqref{eq:uniform_resolvent} makes it injective, hence bijective with
$\|(T_\delta-\omega^2I)^{-1}\|\le C$.
The statement for the adjoint form follows from~\eqref{eq:uniform_infsup}
because the coefficients of $\mathcal A_{\Omega_\delta,\omega^2}$ are
complex symmetric, so that the adjoint problem is the complex conjugate
of the original one.

We first record that $\mathcal D(T_\delta)\subset H^2(\Omega_\delta)$.
This does not follow from the usual convex-domain regularity theory,
which assumes real coefficients, but it follows from the same reflection
that we are about to use. Let $u\in\mathcal D(T_\delta)$, so
$u\in H^1_0(\Omega_\delta)$ and $\widetilde\Delta u\in L^2$. Extend $u$ to
$\Omega^{(3/2)}$ by odd reflection across the four sides and then into
the corners. By~\ref{hyp:saturate} and~\ref{hyp:geom} the coefficients
are invariant under each of these reflections (see the next paragraph),
and odd reflection of an $H^1_0$ function across a flat piece of
boundary preserves the weak formulation; hence the extension is a weak
solution of the same equation on $\Omega^{(3/2)}$, with right-hand side
in $L^2(\Omega^{(3/2)})$. Interior elliptic regularity now applies: on any
open $W$, if $v\in H^1(W)$ and $\widetilde\Delta v\in L^2(W)$ then
$v\in H^2_{\mathrm{loc}}(W)$, because the principal coefficients
$\mathrm{diag}(d_2/d_1,d_1/d_2)$ are $C^\infty$ with bounded
derivatives, depend on $\sigma_1,\sigma_2$ alone, and are uniformly
elliptic in the sense of Lemma~\ref{lemma:bounds1}(3); the propagation
speed does not enter. The
proof is the difference-quotient argument of
Proposition~\ref{prop:resolvent1} run against the test function
$D_k^{-h}(\zeta^2D_k^hv)$ for a cut-off $\zeta$, the extra terms
carrying $\na\zeta$ being absorbed by Young's inequality. This gives
$H^2_{\mathrm{loc}}(\Omega^{(3/2)})$ regularity, and since
$\overline{\Omega_\delta}$ is a compact subset of $\Omega^{(3/2)}$ we
conclude $u\in H^2(\Omega_\delta)$.

Suppose~\eqref{eq:uniform_resolvent} fails. Then there are layer widths
$\delta^{(n)}$ with $\delta^{(n)}_a,\delta^{(n)}_b\to\infty$ and
functions $u_n\in\mathcal D(T_{\delta^{(n)}})\subset
H^2(\Omega_{\delta^{(n)}})\cap H^1_0(\Omega_{\delta^{(n)}})$ with
\begin{equation}\label{eq:weyl_seed}
  \|u_n\|_{L^2(\Omega_{\delta^{(n)}})}=1,
  \qquad
  \bigl\|(-c^2\widetilde\Delta-\omega^2)u_n\bigr\|_{L^2(\Omega_{\delta^{(n)}})}
  \le\frac1n ;
\end{equation}
if for some $n$ the operator $T_{\delta^{(n)}}-\omega^2I$ is not
injective we simply take for $u_n$ a normalised element of its kernel,
for which the second quantity vanishes.

\emph{Reflection.} Extend $u_n$ to $\Omega^{(3/2)}$ by odd reflection
across the four sides of $\p\Omega_\delta$, followed by a second odd
reflection into the four corner regions, as
in~\cite[Fig.~2]{KimPasciak2010a}; the corner values are independent of
the order of the two reflections. Since $u_n\in H^2$ vanishes on
$\p\Omega_\delta$ together with its tangential derivative, the extension
$\tilde u_n$ lies in $H^2(\Omega^{(3/2)})$. The reflection across the
side $x=a+\delta_a$ acts by $(x,y)\mapsto(2(a+\delta_a)-x,\,y)$, so it
leaves $\sigma_2$ untouched, while a point with
$x\ge a+\tfrac12\delta_a$ and its image both satisfy
$\sigma_1=\gamma_a$ by~\ref{hyp:saturate}; both also lie outside
$\overline{B_R}$, because $a+\tfrac12\delta_a>a>R$ by~\ref{hyp:geom},
so that $c\equiv c_\infty$ at both. All coefficients of
$-c^2\widetilde\Delta$ are therefore invariant under the reflection, and
$(-c^2\widetilde\Delta-\omega^2)\tilde u_n$ at a reflected point equals
$\pm(-c^2\widetilde\Delta-\omega^2)u_n$ at its preimage. The same holds
for the horizontal sides and, by composition, for the corner regions,
where both $|x|$ and $|y|$ are large. Since $\Omega^{(3/2)}$ is covered
by finitely many isometric copies of $\Omega_\delta$,
\begin{equation}\label{eq:refl_bound}
  \bigl\|(-c^2\widetilde\Delta-\omega^2)\tilde u_n
  \bigr\|_{L^2(\Omega^{(3/2)})}\le\frac{C}{n}.
\end{equation}

\emph{Cut-off.} Write $\delta_m^{(n)}=\min\{\delta^{(n)}_a,\delta^{(n)}_b\}$
and let $\chi_n\in C_c^\infty(\Omega^{(3/2)})$ with $0\le\chi_n\le1$,
$\chi_n\equiv1$ on $\overline{\Omega_\delta}$,
$\|\na\chi_n\|_{L^\infty}\le C/\delta_m^{(n)}$ and
$\|D^2\chi_n\|_{L^\infty}\le C/(\delta_m^{(n)})^2$. Set
$w_n=\chi_n\tilde u_n$, extended by zero; then $w_n\in H^2(\R^2)$ and
$\|w_n\|_{L^2(\R^2)}\ge\|u_n\|_{L^2(\Omega_{\delta^{(n)}})}=1$. Moreover
\begin{equation}
  \bigl\|(-c^2\widetilde\Delta-\omega^2)w_n\bigr\|_{L^2(\R^2)}
  \le\bigl\|[\,c^2\widetilde\Delta,\chi_n]\,\tilde u_n\bigr\|_{L^2}
   +\bigl\|\chi_n(-c^2\widetilde\Delta-\omega^2)\tilde u_n\bigr\|_{L^2},
\end{equation}
the second term being $O(1/n)$ by~\eqref{eq:refl_bound}. In the
commutator only terms carrying at least one derivative of $\chi_n$
survive, so, the coefficients $d_j^{\pm1}$, $d_j'$ and $c$ being
bounded,
\begin{equation}
  \bigl\|[\,c^2\widetilde\Delta,\chi_n]\,\tilde u_n\bigr\|_{L^2}
  \le C\left(\frac{1}{\delta_m^{(n)}}
  +\frac{1}{(\delta_m^{(n)})^2}\right)\|\tilde u_n\|_{H^1}.
\end{equation}
Testing the weak form with $u_n$ and using the coercivity of
$\mathcal A_{\Omega_\delta,w_0}$ from Lemma~\ref{lm:well-possed} gives
$\|u_n\|_{H^1}\le C\bigl(\|u_n\|_{L^2}
+\|(-c^2\widetilde\Delta-\omega^2)u_n\|_{L^2}\bigr)\le C$, a bound which
survives the reflection. Since $\delta_m^{(n)}\to\infty$, both terms
tend to zero, and therefore
$\|(-c^2\widetilde\Delta-\omega^2)w_n\|_{L^2(\R^2)}\to0$ while
$\|w_n\|_{L^2(\R^2)}\ge1$. This contradicts~\eqref{eq:plane_resolvent}
and proves~\eqref{eq:uniform_resolvent}.

Finally, a nonzero $u\in\ker(T_\delta-\omega^2I)$ would
violate~\eqref{eq:uniform_infsup}, so $\omega^2\notin\sigma_p(T_\delta)$,
and the last assertion is the definition of $C_\delta(\omega)$.
\end{proof}

\begin{rmk}
Proposition~\ref{prop:uniform} supersedes
Proposition~\ref{prop:truncated_wellposed}: for layers thicker than
$\delta_0(\omega)$ there is no exceptional set of frequencies at all.
We have kept Proposition~\ref{prop:truncated_wellposed} because it is
proved by elementary means, holds for every $\delta$, and does not
require~\ref{hyp:geom} or~\ref{hyp:saturate}. What is not claimed here
is uniformity in $\omega$: the threshold $\delta_0$ and the constant
$C_\star$ both depend on the frequency, exactly as
in~\cite{KimPasciak2010a,KimPasciak2010b} for a homogeneous background.
\end{rmk}

\section{Exponential convergence of the truncated PML solution}
\label{sec:convergence}

Proposition~\ref{prop:truncated_wellposed} guarantees that the truncated
problem~\eqref{eq:PML_truncated_problem} is solvable for almost every
$\omega>0$, but it says nothing about how well its solution $u_\delta$
approximates the physical solution $u$ of~\eqref{eq:real_helmholtz}
produced by Proposition~\ref{prop:physical}. The purpose of this
section is to close that gap: we prove that the error decays
exponentially in the product $\gamma\delta$ of absorption strength and
layer thickness.

Throughout this section we keep hypotheses
\ref{hyp:support}--\ref{hyp:saturate} and
$0<\gamma_a\leq\gamma_b=\sigma_M<1$.

\begin{rmk}\label{rmk:why_min}
The requirement $R<\min\{a,b\}$, rather than the weaker $R\le a$, is
not cosmetic. The Cartesian stretching acts on $x$ and $y$
\emph{separately}, and the sign of the imaginary part produced by the
stretching is controlled by the sign of $(x-y_1)\tau_1(x)$ and of
$(y-y_2)\tau_2(y)$ for $\y=(y_1,y_2)\in\p B_R$
(see~\eqref{eq:imrho2} below). Both products are non-negative for every
$\y\in\p B_R$ precisely when $|y_1|\le R\le a$ \emph{and}
$|y_2|\le R\le b$. If $R>b$, points $\y$ near the top of $\p B_R$ have
$y_2>b$, and for $\x$ in the top layer the second product becomes
negative; the complex stretching then \emph{amplifies} instead of
damping the corresponding contribution, and the whole argument below
collapses. Strictness is used to guarantee $|\x|>R$ on a neighbourhood
of $\p\Omega$, where the two definitions of the extension~\eqref{eq:def_hatu}
must agree.
\end{rmk}

\subsection{Uniform bounds for Hankel functions in the first quadrant}

The decay of the analytically continued solution is governed entirely by
the following elementary but sharp bounds, which we could not locate in
the form needed and therefore prove.

\begin{lemma}\label{lem:hankel}
Let $\zeta\in\C\setminus\{0\}$ with $0\le\Arg\zeta\le\pi/2$. Then
\begin{equation}\label{eq:hankel0}
  \bigl|H_0^{(1)}(\zeta)\bigr|
  \;\le\; \sqrt{\frac{2}{\pi|\zeta|}}\;e^{-\Im\zeta},
\end{equation}
\begin{equation}\label{eq:hankel1}
  \bigl|H_1^{(1)}(\zeta)\bigr|
  \;\le\; \sqrt{\frac{2}{\pi|\zeta|}}
          \left(1+\frac{2}{\sqrt{2\pi|\zeta|}}\right)e^{-\Im\zeta}.
\end{equation}
\end{lemma}

\begin{proof}
We use the classical integral representation
\begin{equation}\label{eq:watson}
  H_\nu^{(1)}(\zeta)=\sqrt{\frac{2}{\pi\zeta}}\,
  \frac{e^{\ri\left(\zeta-\frac{\nu\pi}{2}-\frac{\pi}{4}\right)}}
       {\Gamma\!\left(\nu+\tfrac12\right)}
  \int_0^\infty e^{-t}\,t^{\nu-\frac12}
  \left(1+\frac{\ri t}{2\zeta}\right)^{\nu-\frac12}dt,
\end{equation}
valid for $\Re\nu>-\tfrac12$ and $-\tfrac{\pi}{2}<\Arg\zeta<\tfrac{3\pi}{2}$
\cite[\S6.12]{watson}; in particular it is valid on the closed first
quadrant. Since $|e^{\ri\zeta}|=e^{-\Im\zeta}$ and
$\bigl|\sqrt{2/(\pi\zeta)}\bigr|=\sqrt{2/(\pi|\zeta|)}$, everything
reduces to estimating the integral.

Let $0\le\Arg\zeta\le\pi/2$ and $t>0$. Then
$\Arg\bigl(\ri t/(2\zeta)\bigr)=\tfrac{\pi}{2}-\Arg\zeta\in[0,\tfrac{\pi}{2}]$,
so the complex number $w_t:=1+\ri t/(2\zeta)$ satisfies $\Re w_t\ge1$ and
therefore
\begin{equation}\label{eq:wt}
  1\le|w_t|\le 1+\frac{t}{2|\zeta|},
  \qquad \Arg w_t\in\bigl[0,\tfrac{\pi}{2}\bigr).
\end{equation}
Because the exponent $\nu-\tfrac12$ is real, $|w_t^{\nu-1/2}|=|w_t|^{\nu-1/2}$.

For $\nu=0$ the exponent is $-\tfrac12<0$, so by~\eqref{eq:wt}
$|w_t^{-1/2}|\le1$ and
\begin{equation}
  \left|\int_0^\infty e^{-t}t^{-\frac12}w_t^{-\frac12}dt\right|
  \le\int_0^\infty e^{-t}t^{-\frac12}dt=\Gamma\!\left(\tfrac12\right),
\end{equation}
which cancels the factor $\Gamma(1/2)$ in~\eqref{eq:watson} and
yields~\eqref{eq:hankel0}.

For $\nu=1$ the exponent is $\tfrac12>0$; using
$(1+\lambda)^{1/2}\le1+\lambda^{1/2}$ for $\lambda\ge0$ we get from~\eqref{eq:wt}
\begin{equation}
  \left|\int_0^\infty e^{-t}t^{\frac12}w_t^{\frac12}dt\right|
  \le\int_0^\infty e^{-t}t^{\frac12}
     \left(1+\sqrt{\frac{t}{2|\zeta|}}\right)dt
  =\Gamma\!\left(\tfrac32\right)+\frac{\Gamma(2)}{\sqrt{2|\zeta|}} .
\end{equation}
Dividing by $\Gamma(3/2)=\sqrt\pi/2$ and using $\Gamma(2)=1$ gives the
factor $1+2/\sqrt{2\pi|\zeta|}$, which is~\eqref{eq:hankel1}.
\end{proof}

\begin{rmk}
Estimate~\eqref{eq:hankel0} is sharp: the ratio between its two sides
tends to $1$ as $|\zeta|\to\infty$ along the positive real axis. Note
also that both bounds are uniform up to the boundary rays
$\Arg\zeta=0$ and $\Arg\zeta=\pi/2$, which is what makes them usable at
the interface $\p\Omega$, where the stretching degenerates.
\end{rmk}

\subsection{Geometry of the Cartesian stretching}

The following lemma is the heart of the matter, and the place where the
Cartesian (as opposed to radial) nature of the layer must be confronted
directly.

\begin{lemma}\label{lem:geometry}
Assume~\ref{hyp:geom}. Let $\x=(x,y)\in\overline{\Omega^{(2)}}$ with
$|\x|>R$, let $\y=(y_1,y_2)\in\p B_R$, and set
\begin{equation}\label{eq:def_z}
  \mathbf z=T(\x)-\y=(\x-\y)+\ri\,\boldsymbol\tau(\x)\in\C^2 .
\end{equation}
Then:
\begin{enumerate}
  \item\label{it:geo1} $\ds\Im\rho(\mathbf z)^2
        =2\bigl[(x-y_1)\tau_1(x)+(y-y_2)\tau_2(y)\bigr]\ge0$,
        with equality if and only if $\x\in\overline\Omega$;
  \item\label{it:geo2} $\rho(\mathbf z)^2\notin(-\infty,0]$; hence
        $\rho(\mathbf z)$ is well defined and
        $0\le\Arg\rho(\mathbf z)\le\pi/2$;
  \item\label{it:geo3} $|\rho(\mathbf z)|\le\Lambda_\delta$;
  \item\label{it:geo4} for every $s\in(0,1)$ and every
        $\x\in\overline{\Omega^{(2)}}\setminus\Omega^{(s)}$,
        \begin{equation}\label{eq:Theta_bound}
          \Im\rho(\mathbf z)\;\ge\;\Theta(s)\;>\;0 .
        \end{equation}
\end{enumerate}
\end{lemma}

\begin{proof}
\emph{\ref{it:geo1}.} By~\eqref{eq:def_z},
\begin{equation}\label{eq:imrho2}
  \rho(\mathbf z)^2=\bigl|\x-\y\bigr|^2-\bigl|\boldsymbol\tau(\x)\bigr|^2
  +2\ri\,(\x-\y)\cdot\boldsymbol\tau(\x),
\end{equation}
which gives the stated formula for the imaginary part. We check its
sign termwise. By construction $\sigma_1\ge0$ and $\sigma_1\equiv0$ on
$[-a,a]$, so $\tau_1(x)=0$ for $|x|\le a$, while for $|x|>a$ we have
$\tau_1(x)=\operatorname{sgn}(x)\int_a^{|x|}\sigma_1>0$ in absolute
value and $\operatorname{sgn}\tau_1(x)=\operatorname{sgn}(x)$. By
\ref{hyp:geom}, $|y_1|\le R<a$, hence for $|x|>a$ the number $x-y_1$ has
the sign of $x$ as well and
\begin{equation}\label{eq:term1}
  (x-y_1)\tau_1(x)=|x-y_1|\,|\tau_1(x)|\ \ge\ 0 ,
\end{equation}
the same identity holding trivially (both sides zero) when $|x|\le a$.
The analogous statement for the second term uses $|y_2|\le R<b$. Adding,
$\Im\rho(\mathbf z)^2\ge0$, with equality if and only if
$\tau_1(x)=\tau_2(y)=0$, i.e. $|x|\le a$ and $|y|\le b$.

\emph{\ref{it:geo2}.} If $\x\notin\overline\Omega$ then
$\Im\rho(\mathbf z)^2>0$ by~\ref{it:geo1}, so $\rho(\mathbf z)^2$ is not
a non-positive real. If $\x\in\overline\Omega$ then
$\boldsymbol\tau(\x)=0$ and $\rho(\mathbf z)^2=|\x-\y|^2>0$, because
$|\x|>R=|\y|$. In both cases $\Arg\rho(\mathbf z)^2\in[0,\pi]$, and the
principal square root gives $\Arg\rho(\mathbf z)\in[0,\pi/2]$.

\emph{\ref{it:geo3}.} Writing $|\mathbf z|^2=|z_1|^2+|z_2|^2$ we have
$|\rho(\mathbf z)|^2=|z_1^2+z_2^2|\le|\mathbf z|^2$, and since the real
and imaginary parts of $\mathbf z$ are orthogonal components of a vector
of $\R^4$,
\begin{equation}
  |\mathbf z|=\sqrt{|\x-\y|^2+|\boldsymbol\tau(\x)|^2}
  \le|\x-\y|+|\boldsymbol\tau(\x)| .
\end{equation}
On $\overline{\Omega^{(2)}}$ one has
$|\x-\y|\le|\x|+R\le\sqrt{(a+2\delta_a)^2+(b+2\delta_b)^2}+R$ and, since
$0\le\sigma_1\le\gamma_a$ vanishes on $[-a,a]$,
$|\tau_1(x)|\le\gamma_a\cdot 2\delta_a$ for $|x|\le a+2\delta_a$, and
similarly $|\tau_2(y)|\le2\gamma_b\delta_b$, whence
$|\boldsymbol\tau(\x)|\le2\sigma_M\sqrt{\delta_a^2+\delta_b^2}$.
Adding gives $|\rho(\mathbf z)|\le|\mathbf z|\le\Lambda_\delta$.

\emph{\ref{it:geo4}.} Let $\x\in\overline{\Omega^{(2)}}\setminus\Omega^{(s)}$;
then $|x|\ge a+s\delta_a$ or $|y|\ge b+s\delta_b$. Suppose the former
(the other case is symmetric). Then $|x-y_1|\ge a+s\delta_a-R>0$ and, since $\sigma_1\ge0$,
\begin{equation}
  |\tau_1(x)|\ =\ \int_a^{|x|}\sigma_1(t)\,dt
  \ \ge\ \int_a^{a+s\delta_a}\sigma_1(t)\,dt\ =\ \Sigma_1(s).
\end{equation}
Since the second term in~\ref{it:geo1} is non-negative,
\begin{equation}\label{eq:B_lower}
  \Im\rho(\mathbf z)^2\ \ge\ 2\,\Sigma_1(s)\bigl(a+s\delta_a-R\bigr).
\end{equation}
On the other hand, writing $\rho=\rho(\mathbf z)$ we have
$\Im(\rho^2)=2\,\Re\rho\,\Im\rho$ and, by~\ref{it:geo2},
$\Re\rho\ge0$, so
\begin{equation}
  \Im\rho=\frac{\Im(\rho^2)}{2\,\Re\rho}\ \ge\ \frac{\Im(\rho^2)}{2|\rho|}
  \ \ge\ \frac{\Im(\rho^2)}{2\Lambda_\delta},
\end{equation}
using~\ref{it:geo3}. Combining with~\eqref{eq:B_lower} and taking the
minimum over the two cases gives exactly~\eqref{eq:Theta_bound}.
\end{proof}

\begin{rmk}\label{rmk:no_fourier_bessel}
It is tempting to obtain the analytic continuation from the
Fourier--Bessel expansion
\begin{equation}\label{eq:fourier_bessel}
  u=\sum_{n=-\infty}^{\infty}a_nH^{(1)}_n(kr)e^{\ri n\vartheta},
  \qquad r>R,
\end{equation}
by continuing each mode through $r\mapsto\tilde r$. That route is natural for a
\emph{radial} PML, where the stretching acts on $r$ alone, but it is not
available here: the Cartesian map $T$ does not preserve the polar
variables, so $T(\x)$ does not correspond to any single complex value of
$r$, and the modewise continuation carries no information about
$\rho(T(\x)-\y)$. Lemma~\ref{lem:geometry} replaces that argument by a
direct estimate on the complexified distance, which is intrinsically
adapted to the separable Cartesian stretching.
\end{rmk}

\subsection{The exact PML extension of the physical solution}

\begin{lemma}\label{lem:continuation}
Assume~\ref{hyp:support}--\ref{hyp:geom} and let $u$ be the solution
of~\eqref{eq:real_helmholtz} given by
Proposition~\ref{prop:physical}. Define, for
$\x\in\overline{\Omega^{(2)}}$ with $|\x|>R$,
\begin{equation}\label{eq:def_U}
  \mathcal U(\x)=\int_{\p B_R}
  \left[u(\y)\,\frac{\p\Phi}{\p\n(\y)}\bigl(T(\x)-\y\bigr)
        -\frac{\p u}{\p \n}(\y)\,\Phi\bigl(T(\x)-\y\bigr)\right]ds(\y),
\end{equation}
where $\n$ is the unit normal on $\p B_R$ pointing away from the origin.
Then:
\begin{enumerate}
  \item\label{it:cont1} $\mathcal U$ is well defined and of class
        $C^\infty$ on $\{\x\in\overline{\Omega^{(2)}}:|\x|>R\}$;
  \item\label{it:cont2} $\mathcal U(\x)=u(\x)$ whenever
        $\x\in\overline\Omega$ and $|\x|>R$;
  \item\label{it:cont3} $\widetilde\Delta\,\mathcal U+k^2\mathcal U=0$ on
        $\Omega^{(2)}\setminus\overline{B_R}$;
  \item\label{it:cont4} for every $s\in(0,1)$ and every
        $\x\in\overline{\Omega^{(2)}}\setminus\Omega^{(s)}$,
        \begin{equation}\label{eq:U_pointwise}
        \begin{split}
          |\mathcal U(\x)|\;\le\;
          \frac{\sqrt{2\pi R}}{4}&\sqrt{\frac{2}{\pi k\,\Theta(s)}}\,e^{-k\,\Theta(s)}
          \Bigl[\Bigl\|\frac{\p u}{\p\n}\Bigr\|_{L^2(\p B_R)}\\
          &+\frac{k\Lambda_\delta}{\Theta(s)}
           \Bigl(1+\frac{2}{\sqrt{2\pi k\,\Theta(s)}}\Bigr)
           \|u\|_{L^2(\p B_R)}\Bigr] .
        \end{split}
        \end{equation}
\end{enumerate}
\end{lemma}

\begin{proof}
By~\ref{hyp:geom} we have $f=0$ and $c\equiv c_\infty$ on
$\R^2\setminus\overline{B_R}$, so $u$ is a radiating solution of the
constant-coefficient Helmholtz equation $\Delta u+k^2u=0$ there, and by
elliptic regularity $u$ is real analytic on the open set
$\R^2\setminus\overline{B_R}$. Regularity \emph{up to} $\p B_R$ is not
claimed: for a merely bounded $c$ that oscillates against the interface,
$u$ need be no better than $H^2$ across it. What is available, and all
that is used, is
Proposition~\ref{prop:physical}(\ref{it:phys_exist}) together with the
trace theorem: $u\in H^2_{\mathrm{loc}}(\R^2)$ gives
$u|_{\p B_R}\in H^{3/2}(\p B_R)$ and
$\p u/\p\n|_{\p B_R}\in H^{1/2}(\p B_R)$, so that the Cauchy data lie in
$L^2(\p B_R)$ and the Green representation of a radiating solution holds
in the $H^{1/2}\times H^{-1/2}$ duality~\cite[Thm.~2.5]{coltonkress}.

\emph{\ref{it:cont1}.} By Lemma~\ref{lem:geometry}(\ref{it:geo2}), for
every $\x$ in the indicated set and every $\y\in\p B_R$ the quantity
$\rho(T(\x)-\y)$ lies in the closed first quadrant and is nonzero, hence
away from the branch cut and from the logarithmic singularity of
$H_0^{(1)}$. Therefore $\mathbf z\mapsto\Phi(\mathbf z)$ is holomorphic
on a neighbourhood of the compact set
$\{T(\x)-\y\}$, the integrand in~\eqref{eq:def_U} is smooth in $\x$ and
continuous in $\y$, and differentiation under the integral sign is
legitimate.

\emph{\ref{it:cont2}.} If $\x\in\overline\Omega$ then
$\boldsymbol\tau(\x)=0$ and $T(\x)=\x$, so $\Phi(T(\x)-\y)$ reduces to
the classical outgoing fundamental solution
$\tfrac{\ri}{4}H_0^{(1)}(k|\x-\y|)$ and~\eqref{eq:def_U} is exactly
Green's representation formula for a radiating solution in the exterior
of $B_R$ \cite[Thm.~2.5]{coltonkress}, evaluated at
$\x\in\R^2\setminus\overline{B_R}$.

\emph{\ref{it:cont3}.} Fix $\y\in\p B_R$. As a function of
$\mathbf z\in\C^2$, $\Phi$ is holomorphic near $T(\x)-\y$ and satisfies
$\p_{z_1}^2\Phi+\p_{z_2}^2\Phi+k^2\Phi=0$, this being the holomorphic
continuation of the identity valid for real arguments. By the chain rule
for the holomorphic map $T$, $\p_x=d_1(x)\,\p_{z_1}$ and
$\p_y=d_2(y)\,\p_{z_2}$ acting on $\mathbf z\mapsto\Phi(T(\x)-\y)$, and
hence
\begin{equation}
  \frac{1}{d_1}\p_x \left(\frac{1}{d_1}\p_x \right)
  +\frac{1}{d_2}\p_y \left(\frac{1}{d_2}\p_y \right)
  \;=\;\widetilde\Delta
\end{equation}
applied to $\Phi(T(\cdot)-\y)$ reproduces
$(\p_{z_1}^2+\p_{z_2}^2)\Phi$ evaluated at $T(\x)-\y$. Thus
$\widetilde\Delta\,\Phi(T(\cdot)-\y)+k^2\Phi(T(\cdot)-\y)=0$, and the
same holds for $\p\Phi/\p\n(\y)$ and, after integration in $\y$, for
$\mathcal U$.

\emph{\ref{it:cont4}.} Write $\rho=\rho(T(\x)-\y)$ and
$\Theta=\Theta(s)$. By Lemma~\ref{lem:geometry},
$\Arg(k\rho)\in[0,\pi/2]$, $\Im(k\rho)=k\Im\rho\ge k\Theta$ and
$k\Theta\le k|\rho|\le k\Lambda_\delta$. Lemma~\ref{lem:hankel} then gives
\begin{equation}\label{eq:Phi_bound}
  |\Phi(T(\x)-\y)|
  =\frac14\bigl|H^{(1)}_0(k\rho)\bigr|
  \le\frac14\sqrt{\frac{2}{\pi k\Theta}}\,e^{-k\Theta}.
\end{equation}
For the double-layer kernel, $\p_{z_j}\Phi(\mathbf z)
=-\tfrac{\ri k}{4}H_1^{(1)}(k\rho)\,z_j/\rho$, because
$\left(H_0^{(1)}\right)'=-H_1^{(1)}$ and $\p_{z_j}\rho=z_j/\rho$. Using
$|z_j|\le|\mathbf z|\le\Lambda_\delta$ and
$|\rho|\ge\Im\rho\ge\Theta$ we obtain
\begin{equation}\label{eq:dPhi_bound}
  \left|\frac{\p\Phi}{\p\n(\y)}\bigl(T(\x)-\y\bigr)\right|
  \le\frac{k}{4}\,\frac{\Lambda_\delta}{\Theta}\,
     \bigl|H_1^{(1)}(k\rho)\bigr|
  \le\frac{k\Lambda_\delta}{4\,\Theta}
     \sqrt{\frac{2}{\pi k\Theta}}
     \left(1+\frac{2}{\sqrt{2\pi k\Theta}}\right)e^{-k\Theta}.
\end{equation}
Inserting~\eqref{eq:Phi_bound} and~\eqref{eq:dPhi_bound}
into~\eqref{eq:def_U} and applying the Cauchy--Schwarz inequality on
$\p B_R$, whose length is $2\pi R$, yields~\eqref{eq:U_pointwise}. We
use $L^2$ rather than $L^\infty$ norms of the Cauchy data precisely
because the latter need not be finite when $c$ is merely bounded, as
noted at the beginning of this proof; the two bounds differ only by the
factor $\sqrt{2\pi R}$.
\end{proof}

We can now glue $u$ and its stretched continuation into a single
function on $\Omega_\delta$ which solves the PML equation exactly.

\begin{pro}\label{prop:hat_u}
Assume~\ref{hyp:support}--\ref{hyp:geom} and define
\begin{equation}\label{eq:def_hatu}
  \hat u(\x)=
  \begin{cases}
    u(\x), & \x\in\overline\Omega,\\[2pt]
    \mathcal U(\x), & \x\in\overline{\Omega^{(2)}}\setminus\Omega .
  \end{cases}
\end{equation}
Then $\hat u\in H^1(\Omega^{(2)})$, $\hat u$ is smooth on
$\overline{\Omega^{(2)}}\setminus\overline{B_R}$, and
\begin{equation}\label{eq:hatu_weak}
  \mathcal A_{\Omega_\delta,\omega^2}(\hat u,v)=(Jf,v)_{\Omega_\delta}
  \qquad\text{for every }v\in H^1_0(\Omega_\delta).
\end{equation}
\end{pro}

\begin{proof}
\emph{Consistency and smoothness across $\p\Omega$.} By~\ref{hyp:geom},
$|\x|\ge\min\{a,b\}>R$ for every $\x\in\p\Omega$, so there is an open
neighbourhood $N$ of $\p\Omega$ contained in
$\{|\x|>R\}\cap\Omega^{(2)}$. On $N\cap\overline\Omega$ we have
$\mathcal U=u$ by Lemma~\ref{lem:continuation}(\ref{it:cont2}), so the two
branches of~\eqref{eq:def_hatu} agree there and $\hat u=\mathcal U$ on
all of $N$; by Lemma~\ref{lem:continuation}(\ref{it:cont1}), $\hat u$ is
$C^\infty$ on $N$. Since $u\in H^1(\Omega)$ and $\mathcal U$ is smooth
on the compact set $\overline{\Omega^{(2)}}\setminus\Omega$, we get
$\hat u\in H^1(\Omega^{(2)})$.

\emph{The equation.} Let $v\in C_c^\infty(\Omega_\delta)$. On $\Omega$
we have $\sigma_1=\sigma_2=0$, hence $d_1=d_2=1$, $J=c^{-2}$ and
$\widetilde\Delta=\Delta$, so the integrand of
$\mathcal A_{\Omega,\omega^2}(\hat u,v)-(Jf,v)_\Omega$ is the weak form of
$-\Delta u-\omega^2c^{-2}u=c^{-2}f$, i.e. of $-c^2\Delta u-\omega^2u=f$,
which holds by Proposition~\ref{prop:physical}. On
$\Omega_\delta\setminus\overline\Omega$ we have $f=0$, $c\equiv c_\infty$,
and $\omega^2c_\infty^{-2}=k^2$, so the corresponding integrand is the
weak form of $-\widetilde\Delta\mathcal U-k^2\mathcal U=0$, which holds by
Lemma~\ref{lem:continuation}(\ref{it:cont3}). Integrating by parts on each
of the two subdomains, the interface contributions along $\p\Omega$
cancel because $\hat u$ is $C^1$ across $\p\Omega$, as shown above, and
the coefficients $d_2/d_1$, $d_1/d_2$ are continuous there. No
contribution arises on $\p\Omega_\delta$ since $v$ has compact support.
Density of $C_c^\infty(\Omega_\delta)$ in $H^1_0(\Omega_\delta)$
gives~\eqref{eq:hatu_weak}.
\end{proof}

\begin{rmk}
The function $\hat u$ does \emph{not} belong to $H^1_0(\Omega_\delta)$:
it fails the Dirichlet condition on $\p\Omega_\delta$ by precisely the
exponentially small amount estimated
in~\eqref{eq:U_pointwise}. The entire PML error is this boundary
mismatch, and the argument below is a quantitative perturbation
statement to that effect.
\end{rmk}

\subsection{Stability of the truncated problem and a Caccioppoli estimate}

\begin{lemma}\label{lem:infsup}
Let $\omega>0$ with $\omega^2\notin\sigma_p(T_\delta)$. Then the
operator
$\mathcal L_{\omega^2}:H^1_0(\Omega_\delta)\to H^{-1}(\Omega_\delta)$
defined by
$\langle\mathcal L_{\omega^2}u,v\rangle=\mathcal A_{\Omega_\delta,\omega^2}(u,v)$
is an isomorphism. Writing
$C_\delta(\omega)=\|\mathcal L_{\omega^2}^{-1}\|_{H^{-1}\to H^1}$, one has
the inf--sup condition
\begin{equation}\label{eq:infsup}
  \inf_{0\ne u\in H^1_0(\Omega_\delta)}\;
  \sup_{0\ne v\in H^1_0(\Omega_\delta)}\;
  \frac{\bigl|\mathcal A_{\Omega_\delta,\omega^2}(u,v)\bigr|}
       {\|u\|_{H^1}\|v\|_{H^1}}
  \;\ge\;\frac{1}{C_\delta(\omega)}\;>\;0 .
\end{equation}
\end{lemma}

\begin{proof}
Decompose
$\mathcal A_{\Omega_\delta,\omega^2}(u,v)
 =\mathcal A_{\Omega_\delta,w_0}(u,v)-(\omega^2-w_0)(Ju,v)_{\Omega_\delta}$.
By the proof of Proposition~\ref{prop:compact_resolvent}, the form
$\mathcal A_{\Omega_\delta,w_0}$ is continuous and coercive on
$H^1_0(\Omega_\delta)$, so $\mathcal L_{w_0}$ is an isomorphism onto
$H^{-1}(\Omega_\delta)$ by Lax--Milgram. The second term factors as
$H^1_0(\Omega_\delta)\hookrightarrow L^2(\Omega_\delta)
 \xrightarrow{\ J\ }L^2(\Omega_\delta)\hookrightarrow H^{-1}(\Omega_\delta)$,
in which the first embedding is compact
(Rellich--Kondrachov, as in Proposition~\ref{prop:compact_resolvent}) and
the remaining maps are bounded by Lemma~\ref{lemma:bounds1}(2). Hence
$\mathcal L_{\omega^2}$ is a compact perturbation of an isomorphism and
is therefore Fredholm of index $0$.

For injectivity, suppose $u\in H^1_0(\Omega_\delta)$ satisfies
$\mathcal A_{\Omega_\delta,\omega^2}(u,v)=0$ for all
$v\in H^1_0(\Omega_\delta)$. Testing with $v\in C^\infty_c(\Omega_\delta)$
shows that $-c^2\widetilde\Delta u=\omega^2u$ holds in the sense of
distributions; since $\omega^2u\in L^2(\Omega_\delta)$, we get
$\widetilde\Delta u\in L^2(\Omega_\delta)$ and therefore
$u\in\mathcal D(T_\delta)$ with $(T_\delta-\omega^2I)u=0$. By hypothesis
$\omega^2\notin\sigma_p(T_\delta)$, so $u=0$. A Fredholm operator of
index zero with trivial kernel is an isomorphism, and~\eqref{eq:infsup}
is the standard reformulation of the boundedness of
$\mathcal L_{\omega^2}^{-1}$.
\end{proof}

\begin{lemma}[Caccioppoli inequality for the PML operator]\label{lem:caccioppoli}
Let $W\subset\R^2\setminus\overline\Omega$ be open, let
$\chi\in C^\infty_c(W)$ with $0\le\chi\le1$, and let $U\in H^1(W)$
satisfy $\mathcal A_{W,\omega^2}(U,v)=0$ for every $v\in C_c^\infty(W)$,
with $c\equiv c_\infty$ on $W$. Then
\begin{equation}\label{eq:caccioppoli}
  \int_W\chi^2|\na U|^2
  \;\le\;
  16\left(1+\sigma_M^2\right)^{3}
  \left(\|\na\chi\|_{L^\infty}^2+k^2\right)
  \int_{\supp\chi}|U|^2 .
\end{equation}
\end{lemma}

\begin{proof}
The test function $v=\chi^2U$ belongs to $H^1_0(W)$, and
$\overline{\p_xv}=\chi^2\overline{\p_xU}+2\chi(\p_x\chi)\overline U$, with
the analogous identity in $y$. Hence
$\mathcal A_{W,\omega^2}(U,\chi^2U)=0$ reads
\begin{equation}
\begin{split}
  &\int_W\Bigl(\frac{d_2}{d_1}\chi^2|\p_xU|^2+\frac{d_1}{d_2}\chi^2|\p_yU|^2
  -k^2d_1d_2\chi^2|U|^2\Bigr)\\
  &\qquad=-2\int_W\chi\Bigl(\frac{d_2}{d_1}\p_xU\,\p_x\chi\,\overline U
  +\frac{d_1}{d_2}\p_yU\,\p_y\chi\,\overline U\Bigr).
\end{split}
\end{equation}
Taking real parts and using Lemma~\ref{lemma:bounds1}(1) and~(3), together
with $|d_1d_2|\le1+\sigma_M^2$,
\begin{equation}
\begin{split}
  \frac{1}{1+\sigma_M^2}\int_W\chi^2|\na U|^2
  \le\;& 2\sqrt{1+\sigma_M^2}\int_W\chi|\na\chi||\na U||U|\\
  &+k^2\left(1+\sigma_M^2\right)\int_W\chi^2|U|^2 .
\end{split}
\end{equation}
By Young's inequality,
\begin{equation}
  2\sqrt{1+\sigma_M^2}\,\chi|\na\chi||\na U||U|
  \le\frac{1}{2(1+\sigma_M^2)}\chi^2|\na U|^2
  +2\left(1+\sigma_M^2\right)^{2}|\na\chi|^2|U|^2 ;
\end{equation}
absorbing the first term on the left and multiplying by
$2(1+\sigma_M^2)$ gives
\begin{equation}
  \int_W\chi^2|\na U|^2\le
  4\left(1+\sigma_M^2\right)^{3}\|\na\chi\|^2_{L^\infty}\int_{\supp\chi}|U|^2
  +2k^2\left(1+\sigma_M^2\right)^{2}\int_{\supp\chi}|U|^2,
\end{equation}
which implies~\eqref{eq:caccioppoli}.
\end{proof}

\subsection{The convergence theorem}

\begin{thm}\label{thm:convergence}
Assume~\ref{hyp:support}--\ref{hyp:geom} and
$0<\gamma_a\le\gamma_b=\sigma_M<1$. Let $\omega>0$ be such that
$\omega^2\notin\sigma_p(T_\delta)$, let $u$ be the physical solution
given by Proposition~\ref{prop:physical} and let
$u_\delta\in H^1_0(\Omega_\delta)$ be the solution of the truncated
problem~\eqref{eq:PML_truncated_problem} with $w=\omega^2$. Then, for
every $s\in(0,1)$,
\begin{equation}\label{eq:main_estimate}
  \|u_\delta-u\|_{H^1(\Omega)}\;\le\;
  \bigl(1+C_1C_\delta(\omega)\bigr)\,
  \mathcal K_s(\omega,\delta)\,\mathcal N_R(u)\,
  e^{-k\,\Theta(s)},
\end{equation}
where $C_1$ is the continuity constant of
$\mathcal A_{\Omega_\delta,\omega^2}$, $C_\delta(\omega)$ is as in
Lemma~\ref{lem:infsup},
$\mathcal N_R(u)=\|u\|_{L^2(\p B_R)}+\|\p u/\p\n\|_{L^2(\p B_R)}$, which
is finite for every $c$ satisfying~\ref{hyp:support}--\ref{hyp:positive}
by Proposition~\ref{prop:physical}(\ref{it:phys_exist}) and the trace
theorem, and
\begin{equation}\label{eq:def_Ks}
\begin{split}
  \mathcal K_s(\omega,\delta)
  =\;&C_\sigma\left(1+k+\frac{1}{(1-s)\delta_m}\right)
   \bigl|\Omega^{(2)}\bigr|^{1/2}\;
   \frac{\sqrt{2\pi R}}{4}\sqrt{\frac{2}{\pi k\Theta(s)}}\\
  &\times\left[1+\frac{k\Lambda_\delta}{\Theta(s)}
   \left(1+\frac{2}{\sqrt{2\pi k\Theta(s)}}\right)\right],
\end{split}
\end{equation}
with $\delta_m=\min\{\delta_a,\delta_b\}$ and $C_\sigma$ depending only
on $\sigma_M$. In particular $\mathcal K_s$ grows at most polynomially
in $k$ and in $\Theta(s)^{-1}$.
\end{thm}

\begin{proof}
Fix $s\in(0,1)$, set $s'=(1+s)/2$ and choose a cut-off
$\chi\in C_c^\infty\bigl(\Omega^{(2)}\setminus\overline{\Omega^{(s)}}\bigr)$
with $0\le\chi\le1$ and
\begin{equation}\label{eq:chi_props}
  \chi\equiv1 \ \text{ on }\ \overline{\Omega^{(5/4)}}\setminus\Omega^{(s')},
  \qquad
  \supp\chi\subset\overline{\Omega^{(3/2)}}\setminus\Omega^{(s)},
  \qquad
  \|\na\chi\|_{L^\infty}\le\frac{C}{(1-s)\,\delta_m},
\end{equation}
which is possible because the two transition collars have widths
$(s'-s)\delta_j=\tfrac12(1-s)\delta_j$ and $\tfrac14\delta_j$
respectively, both bounded below by a multiple of $(1-s)\delta_m$. Note
that $\chi\equiv1$ on a neighbourhood of $\p\Omega_\delta$, since
$s'<1<5/4$. Set $E=\chi\hat u$, restricted to $\Omega_\delta$; by
Proposition~\ref{prop:hat_u}, $E\in H^1(\Omega_\delta)$. Since
$\chi\equiv1$ on $\p\Omega_\delta$, the function
\begin{equation}
  e:=u_\delta-\hat u+E
\end{equation}
has vanishing trace on $\p\Omega_\delta$, i.e.
$e\in H^1_0(\Omega_\delta)$. For every $v\in H^1_0(\Omega_\delta)$,
by~\eqref{eq:PML_truncated_problem} and~\eqref{eq:hatu_weak},
\begin{equation}
  \mathcal A_{\Omega_\delta,\omega^2}(e,v)
  =\underbrace{\mathcal A(u_\delta,v)-\mathcal A(\hat u,v)}_{=(Jf,v)-(Jf,v)=0}
  +\,\mathcal A_{\Omega_\delta,\omega^2}(E,v)
  =\mathcal A_{\Omega_\delta,\omega^2}(E,v)=:\langle G,v\rangle .
\end{equation}
Although $E\notin H^1_0(\Omega_\delta)$, the map
$v\mapsto\mathcal A_{\Omega_\delta,\omega^2}(E,v)$ is a bounded
antilinear functional on $H^1_0(\Omega_\delta)$, so $G\in
H^{-1}(\Omega_\delta)$ with $\|G\|_{H^{-1}}\le C_1\|E\|_{H^1(\Omega_\delta)}$.
Thus $\mathcal L_{\omega^2}e=G$, and Lemma~\ref{lem:infsup} gives
$\|e\|_{H^1}\le C_\delta(\omega)\,C_1\|E\|_{H^1}$. Consequently
\begin{equation}\label{eq:err_vs_E}
  \|u_\delta-\hat u\|_{H^1(\Omega_\delta)}
  \le\|e\|_{H^1}+\|E\|_{H^1}
  \le\bigl(1+C_1C_\delta(\omega)\bigr)\|E\|_{H^1(\Omega_\delta)} .
\end{equation}
Since $\hat u=u$ on $\Omega$ by construction, the left-hand side
dominates $\|u_\delta-u\|_{H^1(\Omega)}$.

It remains to estimate $\|E\|_{H^1(\Omega_\delta)}$. Set
$W=\Omega^{(2)}\setminus\overline{\Omega^{(s)}}$, an open subset of
$\R^2\setminus\overline\Omega$ on which $\hat u=\mathcal U$ solves the
homogeneous PML equation with $c\equiv c_\infty$
(Lemma~\ref{lem:continuation}(\ref{it:cont3})), and which contains
$\supp\chi$ as a compact subset by~\eqref{eq:chi_props}. Then
\begin{equation}
  \|E\|_{H^1(\Omega_\delta)}\le\|\chi\hat u\|_{L^2(W)}
  +\|(\na\chi)\hat u\|_{L^2(W)}+\|\chi\na\hat u\|_{L^2(W)},
\end{equation}
and Lemma~\ref{lem:caccioppoli}, applicable precisely because
$\chi\in C^\infty_c(W)$, bounds the last term by
$C_\sigma\left(\|\na\chi\|_{L^\infty}+k\right)\|\hat u\|_{L^2(\supp\chi)}$.
Therefore
\begin{equation}\label{eq:E_bound}
  \|E\|_{H^1(\Omega_\delta)}
  \le C_\sigma\left(1+k+\|\na\chi\|_{L^\infty}\right)
      \bigl|\Omega^{(2)}\bigr|^{1/2}
      \sup_{\overline{\Omega^{(2)}}\setminus\Omega^{(s)}}|\mathcal U| .
\end{equation}
Bounding the supremum by
Lemma~\ref{lem:continuation}(\ref{it:cont4}) and inserting
into~\eqref{eq:err_vs_E} yields~\eqref{eq:main_estimate}
with~\eqref{eq:def_Ks}.
\end{proof}

\begin{cor}\label{cor:exp_rate}
Assume the hypotheses of Theorem~\ref{thm:convergence}
and~\ref{hyp:saturate}, and let $\delta_a=\delta_b=\delta$ and
$\gamma_a=\gamma_b=\gamma$. Then, for any fixed $s\in(\tfrac12,1)$ and
all $\delta\ge\delta_1:=(a+b+R)/\bigl(4(1+\sigma_M)\bigr)$,
\begin{equation}\label{eq:cor_rate}
  \|u_\delta-u\|_{H^1(\Omega)}
  \;\le\;C(\omega,\delta)\,
  \exp\left(-\,\frac{\omega}{c_\infty}\,
  \frac{s\bigl(s-\tfrac12\bigr)}{8\,(1+\sigma_M)}\,
  \gamma\,\delta\right),
\end{equation}
where $C(\omega,\delta)$ is the product of the first three factors
in~\eqref{eq:main_estimate}. The error therefore decays exponentially
in $\omega\gamma\delta$, which is the classical PML rate.
\end{cor}

\begin{proof}
By~\ref{hyp:saturate}, $\sigma_1\equiv\gamma$ on
$[a+\tfrac12\delta,\,a+s\delta]$, so
$\Sigma_1(s)\ge\gamma\delta\bigl(s-\tfrac12\bigr)$ for
$s\ge\tfrac12$, and likewise for $\Sigma_2$. Moreover $a>R$
by~\ref{hyp:geom}, whence $a+s\delta-R\ge s\delta$, and likewise in
$b$. Inserting both in~\eqref{eq:def_Theta} gives
\begin{equation}\label{eq:Theta_lower}
  \Theta(s)\;\ge\;\frac{s\bigl(s-\tfrac12\bigr)\,\gamma\,\delta^2}
                       {\Lambda_\delta}.
\end{equation}
Finally
$\Lambda_\delta\le(a+b+R)+4(1+\sigma_M)\delta$
by~\eqref{eq:def_Lambda} and $\sqrt{p^2+q^2}\le p+q$, so that
$\Lambda_\delta\le8(1+\sigma_M)\delta$ for $\delta\ge\delta_1$.
\end{proof}

\begin{rmk}\label{rmk:rate_sharpness}
Two comments on~\eqref{eq:cor_rate}. First, the step
$a+s\delta-R\ge s\delta$ is what makes the exponent grow linearly in
$\delta$; replacing it by the cruder $a+s\delta-R\ge a-R$, which is
what one is tempted to write, produces an exponent that
\emph{saturates} as $\delta\to\infty$, because $\Lambda_\delta$ grows
linearly, and then~\eqref{eq:cor_rate} would assert no decay in the
layer width at all. Second, the constant $s(s-\tfrac12)/8(1+\sigma_M)$
is far from optimal: for the configuration of
Section~\ref{sec:numerics} it is smaller than the exact
$\Theta(s)/\delta$ by a factor of about $5$, and $\Theta(s)$ itself is
in turn smaller than the observed rate by a further factor of about
$12$. The estimate that should be used in practice is therefore
$\Theta(s)$ of~\eqref{eq:def_Theta} itself; \eqref{eq:cor_rate} is
recorded only to exhibit the dependence on $\omega\gamma\delta$. The
same rate structure --- exponential in the product of frequency,
absorption strength and layer width --- appears in the constant-speed
Cartesian analyses of Bramble and Pasciak~\cite{BramblePasciak2013} and
of Wang and Zheng~\cite{WangZheng2026}, and, for a radial layer,
in~\cite{GalkowskiLafontaineSpence2023}.
\end{rmk}

\begin{cor}[Every frequency, uniformly in the layer width]
\label{cor:every_omega}
Let $\omega>0$ be \emph{arbitrary}. Then there is $\delta_0(\omega)$ such that for all
$\delta_a,\delta_b\ge\delta_0(\omega)$ the truncated
problem~\eqref{eq:PML_truncated_problem} with $w=\omega^2$ is uniquely
solvable and
\begin{equation}\label{eq:main_uniform}
  \|u_\delta-u\|_{H^1(\Omega)}\;\le\;
  \bigl(1+C_1C_\star(\omega)\bigr)\,
  \mathcal K_s(\omega,\delta)\,\mathcal N_R(u)\,e^{-k\,\Theta(s)},
\end{equation}
where $C_\star(\omega)$ is the constant of
Proposition~\ref{prop:uniform} and does \emph{not} depend on $\delta$.
\end{cor}

\begin{proof}
By Proposition~\ref{prop:uniform}, for $\delta_a,\delta_b\ge\delta_0$ we
have $\omega^2\notin\sigma_p(T_\delta)$ --- so
Proposition~\ref{prop:truncated_wellposed} applies without any
exceptional set --- and $C_\delta(\omega)\le C_\star(\omega)$. Insert
this in~\eqref{eq:main_estimate}.
\end{proof}

\begin{rmk}[What remains open]\label{rmk:open}
Corollary~\ref{cor:every_omega} removes the two limitations that the
elementary theory of Section~\ref{sec:truncated} left behind: the
exceptional set of frequencies and the dependence of the stability
constant on the size of the computational domain. Two features of the
estimate nevertheless deserve comment.

\emph{(i) Dependence on $\omega$.} Neither the threshold
$\delta_0(\omega)$ nor the constant $C_\star(\omega)$ is shown to be
uniform in $\omega$, and this is the main problem left open by the
present work. It is worth locating it precisely within a literature
that has moved considerably since~\cite{KimPasciak2010a,KimPasciak2010b}.

For a \emph{radial} layer the question is settled: Galkowski,
Lafontaine and Spence~\cite{GalkowskiLafontaineSpence2023} prove that
PML truncation is exponentially accurate with constants explicit in the
frequency, in the generality of black-box scattering, so that even
strongly trapping configurations are covered provided the cut-off
resolvent obeys an exponential bound. Their rate,
$e^{-c(w\tan\theta-C)k}$ in the layer width $w$ and scaling angle
$\theta$, is the exact analogue of our $e^{-k\Theta(s)}$, with
$\tan\theta$ playing the role of $\gamma_M$.

For a \emph{Cartesian} layer the picture is more fragmentary. With a
constant speed and a sound-soft starlike obstacle, Wang and
Zheng~\cite{WangZheng2026} have recently obtained a wavenumber-explicit
inf--sup constant of the optimal order $O(k^{-1})$ for the truncated
uniaxial problem, together with a truncation error
$O\bigl(k^2e^{-k\sigma_0\delta/10}\bigr)$, and their threshold on
$\sigma_0\delta$ is independent of $k$. Their route is instructive: in
place of a perturbation argument they estimate the stretched Green
kernel directly, splitting it dyadically and bounding the resulting
Carleman operators by a weighted Schur test. Two by-products are worth
noting. They work with a \emph{piecewise constant} profile, the case
that~\cite[Rem.~3.5]{KimPasciak2010b} explicitly excludes, and they must
therefore confront the corner singularities of the rectangle, which they
resolve by a piecewise $H^{1+s}$ regularity theory; and the whole-plane
well-posedness they start from is quoted
from~\cite{KimPasciak2010b}, exactly as ours is.

That route does not transfer to a variable speed by perturbation, and it
is worth saying why. Writing $\mathcal A$ for the form with $J=c^{-2}d_1d_2$
and $\mathcal A_\infty$ for the one with $c\equiv c_\infty$, the
difference is $\omega^2\bigl(c_\infty^{-2}-c^{-2}\bigr)d_1d_2(u,v)$,
which in the $k$-weighted norm
$\|v\|^2=\|\na v\|_{L^2}^2+k^2\|v\|^2_{L^2}$ is bounded by
$C\|u\|\,\|v\|$ with $C$ \emph{independent of $k$}, since
$k\|u\|_{L^2}\le\|u\|$. Composed with a solution operator of norm
$O(k)$ this gives $O(k)$, not a contraction: the Neumann series
diverges at high frequency. A compactly supported perturbation of the
speed is therefore invisible to the qualitative Fredholm theory of
Lemma~\ref{lem:reduction} and, at the same time, not small enough for
the quantitative theory. This is the precise sense in which
frequency-explicit bounds for a variable speed require geometry, that
is, a non-trapping hypothesis, rather than perturbation. With a variable
speed the sharpest available statement for a Cartesian layer appears to
be~\cite[Thm.~4.6]{GalkowskiGongGrahamLafontaineSpence2024}, which gives
$O(k^{-\infty})$ rather than exponential decay, under $c\in C^\infty$
and non-trapping.

The result of this paper is therefore complementary rather than
superseded: we obtain an exponential rate for a Cartesian layer over a
variable and merely bounded speed, with a constant that is not explicit
in $\omega$. Closing that gap seems to require a quantitative form of
the limiting absorption principle in place of the qualitative Fredholm
argument used here; the direct kernel estimates
of~\cite{WangZheng2026} and the frequency-explicit bounds for bounded
coefficients of Chaumont-Frelet and
Spence~\cite{ChaumontFreletSpence2023} are the two routes we consider
most promising. Numerically, the spectrum of
$T_\delta$ stays bounded away from the positive real axis in the
resolved part of the spectrum, with no sign of degeneration under
refinement, which is consistent with a threshold $\delta_0$ that is
mild in practice.

\emph{(ii) The factor $\mathcal N_R(u)$.} It is a norm of the Cauchy
data of the physical solution on $\p B_R$, and is controlled by the
local resolvent of $-c^2\Delta$: by the trace theorem,
$\mathcal N_R(u)\le C_R\|u\|_{H^2(B_{R'}\setminus B_{R''})}
\le C_R'\,\Lambda(\omega)\|f\|_{L^2}$ for radii $R''<R<R'$, where
$\Lambda(\omega)$ is governed by
$\bigl\|\chi(-c^2\Delta-(\omega^2+\ri0))^{-1}\chi\bigr\|$. Under the
standing hypotheses~\ref{hyp:support}--\ref{hyp:positive} alone,
$\mathcal N_{R}(u)$ is finite for each $\omega>0$, but nothing is
claimed about its growth: at this regularity the medium may well be
trapping --- a piecewise constant low-velocity inclusion is, by total
internal reflection, as Remark~\ref{rmk:mollify} records --- and
$\Lambda(\omega)$ may then grow exponentially along a sequence of
quasi-resonant frequencies. If, on the other hand, $c\in C^\infty(\R^2)$
satisfies the non-trapping condition~\ref{hyp:morawetz}, the
resolvent bound~\eqref{eq:ntrap_resolvent} of
Burq~\cite{Burq2002} applies --- with the constant made explicit
in~\cite{GalkowskiSpenceWunsch2020} --- and $\Lambda(\omega)$ is then
$O(\omega^{-1})$, so in particular grows at most polynomially in
$\omega$, so that the exponential factor
in~\eqref{eq:cor_rate} dominates for large $\omega$. This is precisely
what the non-trapping analysis of Section~\ref{sec:preliminaries} buys,
and it is the only place where the variability of $c$ enters the size of
the error.
\end{rmk}

\begin{rmk}[Relation to the literature]
For a constant propagation speed, exponential convergence of the
Cartesian PML in $\R^2$ was obtained by Kim and
Pasciak~\cite{KimPasciak2010b}, building on their spectral analysis of
the same operator~\cite{KimPasciak2010a} and on the earlier finite-PML
analysis of Bramble and Pasciak~\cite{BramblePasciak2007}; the convex
and spherical geometries were treated by Lassas and
Somersalo~\cite{LassasSomersalo1998,LassasSomersalo2001}, and adaptive
variants with $\omega$-explicit constants by Chen and
Liu~\cite{ChenLiu2005}. Layered Maxwell media were treated by Duan,
Jiang and Zheng~\cite{DuanJiangZheng2020}.

The closest constant-speed antecedent is Bramble and
Pasciak~\cite{BramblePasciak2013}, who prove inf--sup conditions for the
Cartesian PML on $\R^n$, on the exterior of a scatterer and on the
truncated square, with constants independent of both the layer width and
the absorption strength, and deduce an error of order
$e^{-c\sigma_0M}$ in the region of interest; their analysis covers
$n=2,3$ and both smooth and piecewise constant profiles. That is exactly
the structure of Sections~\ref{sec:allomega}
and~\ref{sec:convergence}, for a homogeneous medium.

Three recent lines of work are also close enough to require comment. Lu,
Lai and Wu~\cite{LuLaiWu2024} establish well-posedness of the uniaxial PML
for a two-layer medium, and obtain the stronger conclusion that the
truncated problem is resonance-free \emph{unconditionally} in the layer
width and absorption strength; the price is a rigid geometry, a single
flat infinite interface, exploited through an explicit Green function.
Jiang, Sun, Sun and Ma~\cite{JiangSunSunMa2024} extend exponential
convergence of the Cartesian PML to a periodic heterogeneous medium,
but by way of homogenisation, so that the estimate concerns the
homogenised coefficient. Galkowski, Gong, Graham, Lafontaine and
Spence~\cite{GalkowskiGongGrahamLafontaineSpence2024} treat a Cartesian
PML with a variable speed, with constants explicit in the wavenumber,
under $c\in C^\infty$ and non-trapping, and obtain $O(k^{-\infty})$
decay.

Against this background the novelty of Theorem~\ref{thm:convergence} is
not that a discontinuous or variable speed is admitted \emph{per se},
but that no structure whatever is imposed on it: neither a single flat
interface, nor periodicity and a separation of scales, nor smoothness.
Hypotheses~\ref{hyp:support}--\ref{hyp:positive} ask only for a
two-sided bound and homogeneity outside a disc, and the conclusion is
exponential rather than super-polynomial. What is \emph{not} obtained
is the wavenumber-explicit control available in each of those three
settings; see Remark~\ref{rmk:open}. Variability enters
through the existence theory of Proposition~\ref{prop:physical} and
through the factor $\mathcal N_R(u)$, while the exponential rate itself
is produced entirely in the homogeneous exterior region and is
therefore identical to the constant-speed one. By
Section~\ref{sec:allomega} the range of admissible frequencies and the
uniformity in $\delta$ also match the constant-speed results.
\end{rmk}

\section{Numerical experiments}
\label{sec:numerics}

We report experiments designed to test the two predictions of
Section~\ref{sec:convergence} that are specific to the variable-speed
setting: that the error decays exponentially in $\delta$ and in
$\omega$, and that the inhomogeneity of $c$ does not degrade the rate.

\subsection{Test medium and discretisation}

We take $\Omega=(-1,1)^2$, $c_\infty=1$, $\gamma_a=\gamma_b=0.9$ and
$\delta_a=\delta_b=\delta$. The transition widths are fixed at
$\delta^{\mathrm t}_a=\delta^{\mathrm t}_b=0.075$, independently of $\delta$, as required by the
convention of Section~\ref{sec:preliminaries};
hypothesis~\ref{hyp:saturate} then holds for every $\delta\ge0.15$, and
all the layer widths used below satisfy this. The medium is the radial
low-velocity inclusion
\begin{equation}\label{eq:test_c}
  c(\varrho)=c_\infty\Bigl(1-\ve\,
  e^{\,s/(s-1)}\Bigr),\qquad s=(\varrho/R)^2,\qquad R=\tfrac12,
\end{equation}
which is $C^\infty$, equals $c_\infty$ for $\varrho\ge R$, and has
$c_{\text{min}}=(1-\ve)c_\infty$ at the origin. We
use $\ve=0.4$, so that $c$ varies by $40\%$ across the region of
interest. This medium is a strict test of
hypothesis~\ref{hyp:morawetz}: with $\x_0=0$ one computes
\begin{equation}
  \max_{\varrho}\ \varrho\,c'(\varrho) = 0.672 > 0,
  \qquad
  \max_{\varrho}\ \frac{\varrho\,c'(\varrho)}{c(\varrho)} = 0.736 < 1,
\end{equation}
so the naive requirement $\na c\cdot(\x-\x_0)\le0$ fails badly, while
\ref{hyp:morawetz} holds with $\mu=0.264$. Figure~\ref{fig:sharpness}
illustrates the corresponding dichotomy: for $\ve=0.4$ the Herglotz
function $\varrho\mapsto\varrho/c(\varrho)$ is increasing and every ray
escapes, whereas for $\ve=0.75$ it is not, \ref{hyp:morawetz} fails, and
a bicharacteristic launched tangentially at the critical radius remains
trapped in an annulus. This confirms numerically that
\ref{hyp:morawetz} is sharp, as asserted in
Remark~\ref{rmk:sharpness}.

\begin{figure}[htbp]
  \centering
  \includegraphics[width=\textwidth]{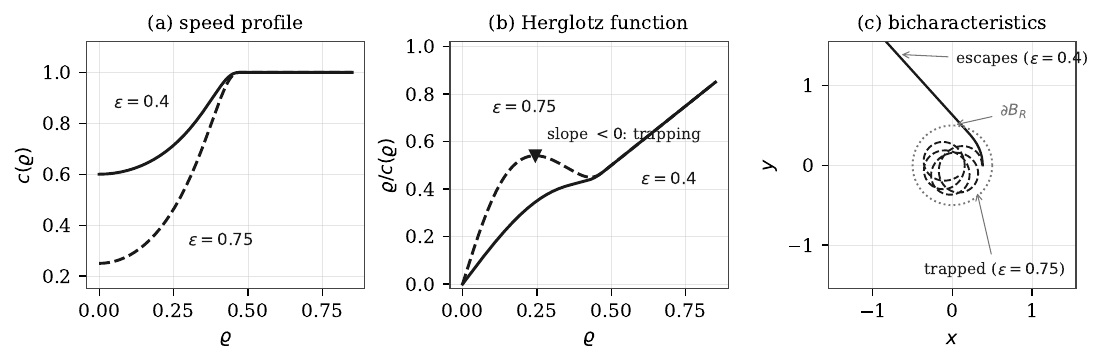}
  \caption{Sharpness of the non-trapping
  condition~\ref{hyp:morawetz} for the medium~\eqref{eq:test_c}.
  (a) Speed profiles. (b) The Herglotz function
  $\varrho/c(\varrho)$: increasing for $\ve=0.4$, non-monotone for
  $\ve=0.75$. (c) Bicharacteristics of $c\,|\xi|$ launched tangentially
  inside $B_R$: the ray escapes when~\ref{hyp:morawetz} holds and is
  trapped when it fails.}
  \label{fig:sharpness}
\end{figure}

The source is a fixed $C^\infty$ bump supported in $B_R$. The truncated
problem~\eqref{eq:PML_truncated_problem} is discretised by the
conservative five-point finite difference scheme associated with the
divergence form of the PML operator, on a uniform grid of mesh size
$h=1/100$, with homogeneous Dirichlet conditions on $\p\Omega_\delta$;
the resulting complex sparse system is solved directly. At $\omega=12$
and $c_{\text{min}}=0.6$ this is about $26$ points per shortest
wavelength, so the discretisation error is far below the PML errors
reported here. Since $u$ is not available in closed form for a variable
speed, we use as reference the solution computed with a very thick
layer, $\delta_{\text{ref}}=1.6$; the quantity tabulated is the relative
$L^2(\Omega)$ error against that reference, which differs from
$\|u_\delta-u\|_{L^2(\Omega)}/\|u\|_{L^2(\Omega)}$ by a term of the
order of the reference error itself, several orders of magnitude smaller
than the entries shown.

\subsection{Decay in the layer thickness}

Table~\ref{tab:delta} and Figure~\ref{fig:convergence}(a) report the
error as a function of $\delta$ for two frequencies. The decay is
clean and geometric, with a constant ratio per increment of $\delta$,
as predicted by Corollary~\ref{cor:exp_rate}.

\begin{table}[htbp]
\centering
\caption{Relative $L^2(\Omega)$ error as a function of the layer
thickness $\delta$, for the medium~\eqref{eq:test_c} with $\ve=0.4$.}
\label{tab:delta}
\begin{tabular}{@{}lcccc@{}}
\toprule
& \multicolumn{2}{c}{$\omega=6$} & \multicolumn{2}{c}{$\omega=10$}\\
\cmidrule(lr){2-3}\cmidrule(lr){4-5}
$\delta$ & error & ratio & error & ratio\\
\midrule
$0.15$ & $4.192\times10^{-1}$ & --- & $2.002\times10^{-1}$ & ---\\
$0.25$ & $1.337\times10^{-1}$ & $0.319$ & $2.960\times10^{-2}$ & $0.148$\\
$0.35$ & $4.151\times10^{-2}$ & $0.310$ & $4.718\times10^{-3}$ & $0.159$\\
$0.45$ & $1.307\times10^{-2}$ & $0.315$ & $7.562\times10^{-4}$ & $0.160$\\
$0.55$ & $4.162\times10^{-3}$ & $0.318$ & $1.209\times10^{-4}$ & $0.160$\\
$0.65$ & $1.344\times10^{-3}$ & $0.323$ & $1.929\times10^{-5}$ & $0.159$\\
\midrule
\multicolumn{1}{@{}l}{observed rate} & \multicolumn{2}{c}{$-11.51$} & \multicolumn{2}{c}{$-18.45$}\\
\multicolumn{1}{@{}l}{rate from $\Theta(s)$, $s=0.9$} & \multicolumn{2}{c}{$-0.94$} & \multicolumn{2}{c}{$-1.57$}\\
\multicolumn{1}{@{}l}{rate from Corollary~\ref{cor:exp_rate}} & \multicolumn{2}{c}{$-0.27$} & \multicolumn{2}{c}{$-0.46$}\\
\bottomrule
\end{tabular}
\end{table}

Two features deserve comment. First, the observed rate scales almost
linearly with $\omega$: the ratio of the two fitted slopes is $1.60$,
against $\omega_2/\omega_1=1.67$. This is what
Corollary~\ref{cor:exp_rate} predicts, since $\Theta(s)$ does not depend
on $\omega$ and the exponent is $\omega\,\Theta(s)/c_\infty$; the small
discrepancy is attributable to the $\omega$-dependence of the prefactor.
Second, the guaranteed rates are rigorous lower bounds but
conservative ones. The exponent $\Theta(s)$ of
Theorem~\ref{thm:convergence}, evaluated directly, is smaller than the
observed rate by a factor of roughly $12$; the further-simplified
constant of Corollary~\ref{cor:exp_rate}, which trades $\Theta(s)$ for
a closed form in $\gamma\delta$, loses another factor of about $3.5$.
Both losses are expected: $\Theta(s)$ is a worst case over all
$\y\in\p B_R$ and all $\x$ in the outer part of the layer
simultaneously, and the denominator $\Lambda_\delta$
of~\eqref{eq:def_Theta} is a crude diameter bound obtained from
$|\rho|\le|\mathbf z|$. The theorem should therefore be read as
establishing the exponential mechanism and its dependence on
$\omega\gamma_M\delta_M$, not as a quantitatively tight prediction of
the constant; see Remark~\ref{rmk:rate_sharpness}.

\begin{figure}[htbp]
  \centering
  \includegraphics[width=\textwidth]{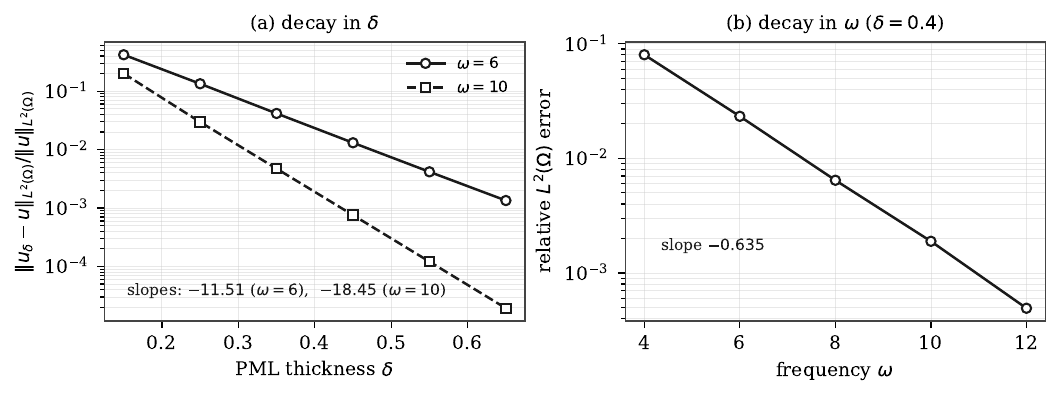}
  \caption{Exponential convergence of the truncated Cartesian PML
  solution for the variable-speed medium~\eqref{eq:test_c}.
  (a) Relative $L^2(\Omega)$ error against the layer thickness $\delta$,
  for $\omega=6$ and $\omega=10$; the slopes of the fitted lines are
  shown. (b) Error against the frequency $\omega$ at fixed
  $\delta=0.4$.}
  \label{fig:convergence}
\end{figure}

\subsection{Decay in the frequency}

Corollary~\ref{cor:exp_rate} also predicts exponential decay in
$\omega$ at fixed geometry, with rate $\Theta(s)/c_\infty$.
Table~\ref{tab:omega} and Figure~\ref{fig:convergence}(b) confirm this
for $\delta=0.4$: the error falls by more than two orders of magnitude
between $\omega=4$ and $\omega=12$, with fitted rate $-0.635$ against
the value $\Theta(0.9)/c_\infty=-0.060$ furnished by
Theorem~\ref{thm:convergence}. The high-frequency
behaviour is thus favourable, as it should be for a method whose
absorption is produced by the oscillation itself.

\begin{table}[htbp]
\centering
\caption{Relative $L^2(\Omega)$ error as a function of $\omega$ at
fixed $\delta=0.4$, for the medium~\eqref{eq:test_c} with $\ve=0.4$.}
\label{tab:omega}
\setlength{\tabcolsep}{4pt}\small
\begin{tabular}{@{}lccccc@{}}
\toprule
$\omega$ & $4$ & $6$ & $8$ & $10$ & $12$\\
\midrule
error & $8.017\times10^{-2}$ & $2.327\times10^{-2}$ &
$6.440\times10^{-3}$ & $1.889\times10^{-3}$ & $4.922\times10^{-4}$\\
\bottomrule
\end{tabular}
\end{table}

\subsection{The spectrum of the truncated operator}

Corollary~\ref{cor:every_omega} asserts that no frequency need be
excluded once $\delta_a,\delta_b\ge\delta_0(\omega)$. It is instructive
to see how $\sigma(T_\delta)$ behaves. We discretised $T_\delta$ with
the same scheme on $\Omega=(-1,1)^2$, $\gamma_a=\gamma_b=0.6$,
$\delta=0.8$, $\delta^{\mathrm t}_a=\delta^{\mathrm t}_b=0.3$, and computed the full spectrum for
three mesh sizes. The relevant quantity is $\min|\Arg\lambda|$, the
angular distance from the positive real axis, restricted to the
resolved part $|\lambda|<40$ of the spectrum, where the number of
eigenvalues is stable ($39$, $37$, $37$) across the three meshes:
\begin{center}
\begin{tabular}{@{}lccc@{}}
\toprule
$h$ & $1/8$ & $1/10$ & $1/12$\\
\midrule
$\min|\Arg\lambda|$ over $|\lambda|<40$ & $0.408$ & $0.405$ & $0.403$\\
\bottomrule
\end{tabular}
\end{center}
The minimum settles near $0.40$ and shows no tendency to degenerate
under refinement. Eigenvalues with $|\Arg\lambda|$ close to zero do
occur, but only at $|\lambda|\sim h^{-2}$, that is, in the unresolved
high-frequency tail of the discretisation; they migrate to larger
$|\lambda|$ as $h$ decreases and are grid artefacts. The experiment is
thus consistent with a threshold $\delta_0(\omega)$ that is mild in
practice, although, as noted in Remark~\ref{rmk:open}, we make no
quantitative claim about it.

\section{Conclusions}
\label{sec:conclusions}

We have analysed the Cartesian perfectly matched layer for the
two-dimensional Helmholtz equation with a variable propagation speed,
a setting not covered by the existing theory, which assumes a
homogeneous or piecewise homogeneous background.

Three ingredients carry the analysis. The first is structural: the
propagation speed enters the Cartesian PML sesquilinear form only
through the zeroth order coefficient $J=c^{-2}d_1d_2$, the principal
part being built from the stretching profiles alone. Every estimate
below therefore sees $c$ only through its two-sided bound, and the
theory applies to discontinuous speeds --- layered media, compact
inclusions --- with the physical problem itself handled by a
Lippmann--Schwinger equation
(Proposition~\ref{prop:physical}). The second is a non-trapping
condition obtained from the flat Morawetz multiplier, which turns out to
be sharp once the two terms of the derivative of the multiplier are kept
together: the controlling quantity is the scale-invariant ratio
$c^{-1}\na c\cdot(\x-\x_0)$, and the resulting
condition~\ref{hyp:morawetz} coincides in the radial case with the
Herglotz condition, whose failure produces a circular orbit. It is not
needed for well-posedness; its role is to identify the regime in which
the Cauchy data of the physical solution stay under control at high
frequency. The third is the analytic continuation of the outgoing solution along the
\emph{separable} Cartesian stretching, controlled by uniform Hankel
bounds on the closed first quadrant and by a lower bound on the
imaginary part of the complexified distance. Together they give
exponential convergence of the truncated PML solution in
$H^1(\Omega)$, at a rate governed by $\omega\gamma_M\delta_M$ and
unaffected by the inhomogeneity, since it is produced entirely in the
homogeneous exterior. The numerical experiments of
Section~\ref{sec:numerics} reproduce this behaviour on a medium with a
$40\%$ low-velocity inclusion.

A fourth ingredient makes the theory complete rather than almost
complete. The essential spectrum does not see a perturbation of the
speed supported in a compact set, once the operator is factored so that
the perturbation acts at zeroth order (Lemma~\ref{lem:reduction}). This
lets us import the spectral analysis of the homogeneous Cartesian layer
of~\cite{KimPasciak2010a} without redeveloping it for a symbol that has
no uniform directional limit at infinity, and, combined with
injectivity and a reflection argument, it removes both the exceptional
set of frequencies and the dependence of the stability constant on the
size of the computational domain.

What remains open is uniformity in $\omega$. Neither the threshold
$\delta_0(\omega)$ nor the constant $C_\star(\omega)$ is shown to be
bounded at high frequency. The frequency-explicit theory is by now
complete for a radial
layer~\cite{GalkowskiLafontaineSpence2023} and, for a Cartesian layer
with a constant speed, has been carried out by Wang and
Zheng~\cite{WangZheng2026}; what is missing is precisely the
combination treated here, a Cartesian layer over a variable and merely
bounded speed, for which the sharpest frequency-explicit statement we
are aware of gives $O(k^{-\infty})$ under a smoothness hypothesis
in~\cite{GalkowskiGongGrahamLafontaineSpence2024}. Replacing the
qualitative Fredholm argument used here by a quantitative limiting
absorption principle, along the lines of the direct kernel estimates
of~\cite{WangZheng2026} or of the frequency-explicit bounds for
bounded coefficients of~\cite{ChaumontFreletSpence2023}, appears to us
the natural next step.

\appendix
\renewcommand{\thethm}{A.\arabic{thm}}
\renewcommand{\theequation}{A.\arabic{equation}}
\setcounter{thm}{0}\setcounter{equation}{0}

\section{The extended-domain operator $\widetilde{L}$}
\label{app:Ltilde}

This appendix records elementary operator-theoretic properties of the
extended-domain PML operator $-c^2\widetilde{\Delta}$ on $\R^2$: that it
is well defined and has non-empty resolvent set, and that the
corresponding weak solutions gain $H^2$ regularity.
These facts are of independent interest and are not needed for the
well-posedness results of Proposition~\ref{prop:physical} and of
Section~\ref{sec:truncated}, which are established directly for the
physical operator $-c^2\Delta$ and for the truncated operator
$T_\delta$, respectively. We include them because the $H^2$ estimate of
Proposition~\ref{prop:resolvent1}, obtained by difference quotients
adapted to the complex stretching, does not seem to be recorded
elsewhere for a variable-speed Cartesian layer, and because the argument
is arranged so as to require no regularity of $c$.

The construction of $\widetilde L$ below is the one place in this paper
where more than~\ref{hyp:positive} is required of the propagation speed:
it proceeds through the substitution $\varphi=\bar Jv$, which calls for
$J$ and $J^{-1}$ to be multipliers of $H^1(\R^2)$, hence for
$c\in W^{1,\infty}(\R^2)$. \emph{Throughout this appendix, and only
here, we therefore assume in addition that $c$ is Lipschitz.} The
regularity statement of Proposition~\ref{prop:resolvent1}, which is the
only result of the appendix used in the body of the paper, is
independent of this and holds under~\ref{hyp:support}
and~\ref{hyp:positive} alone.

The functions $d_1,d_2$ are smooth with bounded derivatives, and $c$ is Lipschitz with $c^{-2}$ bounded above and below; hence $J$ and $J^{-1}$ are bounded and Lipschitz, and the map $v \mapsto \bar{J}v$ is an isomorphism of $H^1(\mathbb{R}^2)$. Therefore, the problem \eqref{eq:PML_extended_problem} is equivalent to the following: given $f \in L^2(\mathbb{R}^2)$ and $w\in\C$, find $u \in H^1(\mathbb{R}^2)$ such that, for any $\varphi \in H^1(\mathbb{R}^2)$,
\begin{equation}\label{eq:weak_form1}
A_{\R^2}(u,\bar{J}^{-1}\varphi) - w(u,\varphi)_{\R^2} = (f,\varphi)_{\R^2}.
\end{equation}

With the aim of characterising values $w\in\C$ for which $\mathcal{A}_{\R^2,w}$ is a continuous and coercive sesquilinear form, an unbounded operator $\widetilde{L}$ is defined from $H^{-1}(\R^2)$ into $H^{-1}(\R^2)$ with domain $H^1(\R^2)$ by means of
 \begin{equation}
\langle\widetilde{L}u,\varphi\rangle = A_{\R^2}(u,\bar{J}^{-1}\varphi), \mbox{ for all $u$ and $\varphi$ in $H^1(\R^2)$}.
\end{equation}
As a first task we will show that $\widetilde{L}$ is well-defined and its resolvent set $\rho(\widetilde{L})$ is not empty. We recall that the resolvent set of the operator $\widetilde{L}$ is formed by those complex numbers $\lambda$ such that $\widetilde{L}-\lambda I$ is one-to-one and $(\widetilde{L}-\lambda I)^{-1}$ is bounded with dense domain in $H^{-1}(\R^2)$. Here $I$ is the canonical inclusion from $H^1(\R^2)$ into $H^{-1}(\R^2)$.

\begin{lemma}\label{lem:Ldefined}
	The operator $\widetilde{L}:H^1(\R^2)\subset H^{-1}(\R^2)\rightarrow H^{-1}(\R^2)$, defined by $\langle \widetilde{L}u,\varphi\rangle =A_{\R^2}(u,\bar{J}^{-1}\varphi)$, is well-defined.
\end{lemma}
\begin{proof}
	Let $u$ and $\varphi$ be in $H^1(\R^2)$. From Lemma~\ref{lemma:bounds1} it follows that there is a positive constant $C=C(\sigma_M,w)$ such that
	\begin{equation}
 \begin{split}
   |\langle\widetilde{L}u,\varphi\rangle|&=|A_{\R^2}(u,\bar{J}^{-1}\varphi)|\\
   &\leq C\|u\|_{H^1}\left\|\bar{J}^{-1}\varphi\right\|_{H^1}.
 \end{split}
	\end{equation}
	Now, $v\mapsto {\bar{J}}^{-1}v$ on $H^{1}(\R^2)$ is an isomorphism, thus there is a positive constant $C_1$ such that $\left\|\bar{J}^{-1}\varphi\right\|_{H^1}\leq C_1\|\varphi\|_{H^1}$, then
     \begin{equation}
         |\langle\widetilde{L}u,\varphi\rangle|\leq CC_1\|u\|_{H^1}\|\varphi\|_{H^1},
     \end{equation} and
	\begin{equation}
	\sup_{\|\varphi\|_{H^1}\neq 0}\frac{|\langle\widetilde{L}u,\varphi\rangle|}{\|\varphi\|_{H^1}}\leq CC_1\|u\|_{H^1}.
	\end{equation}
	Hence $\widetilde{L}u\in H^{-1}(\R^2)$.
\end{proof}

\begin{pro}\label{prop:resolventL}
    The resolvent set $\rho(\widetilde{L})$ is not empty.
\end{pro}
\begin{proof}
A complex number $w_0\in\rho(\widetilde{L})$ if, for any $F\in H^{-1}(\R^2)$,
  there exists a unique $u\in H^1(\R^2)$ such that $(\widetilde{L}-w_0 I)u=F$,
  and there exists a positive constant $C$, independent of $F$, such that
  \begin{equation}\label{eq:boundLinv}
    \|u\|_{H^{1}} \leq C\|F\|_{H^{-1}}.
  \end{equation}
  We take $w_0=-e^{-\ri\theta}$. By Lemma~\ref{lm:well-possed}, for any
  $F\in H^{-1}(\R^2)$ there exists a unique $u\in H^1(\R^2)$ satisfying
  \begin{equation}
    A_{\R^2}(u,\bar{J}^{-1}\varphi) - w_0(u,\varphi)_{\R^2} = \langle F,\varphi
    \rangle
    \qquad \forall\,\varphi\in H^1(\R^2),
  \end{equation}
  together with the estimate $\|u\|_{H^1}\leq C_2^{-1}\|F\|_{H^{-1}}$,
  which gives~\eqref{eq:boundLinv} and shows that $w_0\in\rho(\widetilde{L})$.
\end{proof}

Now, we consider the operator $-c^2\widetilde{\Delta}:H^2(\R^2)\subset L^2(\R^2)\rightarrow L^2(\R^2)$. The following proposition is a first step towards characterising the spectrum of $\widetilde{L}$ and $-c^2\widetilde{\Delta}$.

\begin{pro}\label{prop:resolvent1}
	 The resolvent set $\rho(-c^2\widetilde{\Delta})$ is not empty.
	\end{pro}
\begin{proof}
  To prove that $w_0\in\rho(-c^2\widetilde{\Delta})$ as well, we let
  $f\in L^2(\R^2)$ and let $u\in H^1(\R^2)$ be the unique solution of
  \begin{equation}\label{eq:weak_f}
    \mathcal{A}_{\R^2,w_0}(u,v) = (J f,v)_{\R^2}
    \qquad \forall\,v\in H^1(\R^2),
  \end{equation}
  whose existence, uniqueness, and the bound
  \begin{equation}\label{eq:bound_u_les_f}
      \|u\|_{H^1(\R^2)}\leq C\|f\|_{L^2(\R^2)}
  \end{equation}
  for some constant $C>0$, are guaranteed by Lemma~\ref{lm:well-possed}.
  We shall prove that $u\in H^2(\R^2)$ actually, and that
  $(-c^2\widetilde{\Delta}-w_0 I)u=f$ in $L^2(\R^2)$ with $\|u\|_{H^2}\leq C\|f\|_{L^2}$ for some constant $C>0$.

  The propagation speed is only assumed bounded, so the zeroth-order
  coefficient $J=c^{-2}d_1d_2$ need not be differentiable and difference
  quotients cannot be applied to it. We therefore move that term to the
  right-hand side before differentiating. Fix a real constant
  $\lambda>0$ and set
  \begin{equation}\label{eq:Blambda}
    \mathcal{B}_\lambda(u,v)=A_{\R^2}(u,v)+\lambda(u,v)_{\R^2},
    \qquad
    g = Jf + w_0Ju+\lambda u .
  \end{equation}
  By Lemma~\ref{lemma:bounds1}(2) and~\eqref{eq:bound_u_les_f},
  $g\in L^2(\R^2)$ with
  \begin{equation}\label{eq:bound_g}
    \|g\|_{L^2(\R^2)}\le
    \Bigl(\|J\|_{L^\infty}+\bigl(|w_0|\,\|J\|_{L^\infty}+\lambda\bigr)C\Bigr)
    \|f\|_{L^2(\R^2)},
  \end{equation}
  and~\eqref{eq:weak_f} is equivalent to
  \begin{equation}\label{eq:weak_g}
    \mathcal{B}_\lambda(u,v)=(g,v)_{\R^2}\qquad\forall\,v\in H^1(\R^2).
  \end{equation}
  Moreover, by Lemma~\ref{lemma:bounds1}(3),
  \begin{equation}\label{eq:coercivity_B}
  \begin{split}
    \Re\,\mathcal{B}_\lambda(v,v)&\;\ge\;
    \frac{1}{\sigma_M^2+1}\|\na v\|_{L^2}^2+\lambda\|v\|^2_{L^2}
    \;\ge\;C_2'\|v\|^2_{H^1(\R^2)},\\
    C_2'&=\min\Bigl\{\tfrac{1}{\sigma_M^2+1},\lambda\Bigr\},
  \end{split}
  \end{equation}
  for every $v\in H^1(\R^2)$. The coefficients of $\mathcal{B}_\lambda$
  are $d_2/d_1$, $d_1/d_2$ and the \emph{constant} $\lambda$; by
  Lemma~\ref{lemma:bounds1} the first two belong to $C^\infty(\R^2)$
  with bounded derivatives, because they depend on $\sigma_1,\sigma_2$
  alone and not on $c$. This is what makes the argument below
  independent of the regularity of the propagation speed.

  For $h\neq 0$ and $k=1,2$ define the difference quotients
  \begin{equation*}
    D_k^h u(x,y)=\frac{u((x,y)+h\mathbf{e}_k)-u(x,y)}{h}, \text{ with } \mathbf{e}_1=(1,0) \text{ and } \mathbf{e}_2=(0,1).
  \end{equation*}
  Choose the test function
  \begin{equation}\label{eq:test}
    v = D_k^{-h}(D_k^h u) \;\in\; H^1(\R^2),
  \end{equation}
  where $D_k^{-h}$ denotes the backward difference quotient. For each fixed
  $h\neq0$, $D_k^h u\in H^1(\R^2)$ as well, since translation is an isometry
  of $H^1(\R^2)$.

  For the left-hand side we examine each term of
  $\mathcal{B}_\lambda(u,D_k^{-h}(D_k^hu))$. Recall that
  \begin{equation}\label{eq:formA}
    \mathcal{B}_\lambda(u,v)
    = \int_{\R^2}\!\left(
        \frac{d_2}{d_1}\,\p_x  u\,\overline{\p_x  v}
       +\frac{d_1}{d_2}\,\p_y  u\,\overline{\p_y  v}
       + \lambda\,u\overline v
      \right)dxdy.
  \end{equation}
  The operators of partial differentiation commute with $D_k^h$. Applying
  the discrete integration-by-parts formula and Leibniz's rule for $D_k^h$
  to each term in $\mathcal{B}_\lambda(u,D_k^{-h}D_k^hu)$ gives
\begin{equation}\label{eq:first_term_A}
\begin{split}
     \int_{\R^2}\!
        \frac{d_2}{d_1}\,\p_x  u\,\overline{\p_x   D_k^{-h} D_k^h u}dxdy & = \int_{\R^2}\!
        \frac{d_2}{d_1}\,\p_x  u\,  D_k^{-h}\overline{\p_x  D_k^h  u}dxdy\\
        & = -\int_{\R^2}\! D_k^h  \left(\frac{d_2}{d_1}\,\p_x  u\right)
       \, \overline{\p_x  D_k^h  u}dxdy\\
       & = -\int_{\R^2}\!  \left.\frac{d_2}{d_1}\right|_{(x,y)+h\e_k}
       \,\left|  \p_x  D_k^hu\right|^2dxdy\\
       &\qquad - \int_{\R^2}\!D_k^h\left(  \frac{d_2}{d_1}\right)\,\p_x  u
       \,  \p_x  D_k^h u\,dxdy
\end{split}
\end{equation}
\begin{equation}\label{eq:second_term_A}
\begin{split}
     \int_{\R^2}\!
        \frac{d_1}{d_2}\,\p_y  u\,\overline{\p_y   D_k^{-h} D_k^h u}dxdy
       &= -\int_{\R^2}\!  \left.\frac{d_1}{d_2}\right|_{(x,y)+h\e_k}
       \,\left|  \p_y  D_k^hu\right|^2dxdy\\
       &\qquad - \int_{\R^2}\!D_k^h\left(  \frac{d_1}{d_2}\right)\,\p_y  u
       \,  \p_y  D_k^h u\,dxdy
\end{split}
\end{equation}
\begin{equation}\label{eq:third_term_A}
     \int_{\R^2}\!
        \lambda\, u\,\overline{  D_k^{-h} D_k^h u}dxdy
       = -\int_{\R^2}\!  \lambda
       \,\left|   D_k^hu\right|^2dxdy,
\end{equation}
the commutator contribution being absent in the last identity because
$\lambda$ is constant. This is the single point at which the
reformulation~\eqref{eq:weak_g} is used.
From \eqref{eq:first_term_A}, \eqref{eq:second_term_A} and \eqref{eq:third_term_A}, we get
\begin{equation}\label{eq:A_decompos}
    \mathcal{B}_\lambda(u,D_k^{-h}D_k^hu)=-\mathcal{B}_\lambda^{h,k}(D_k^{h}u,D_k^hu)-\mathcal{C}^{h,k}(u,D_k^hu),
\end{equation}
where
\begin{equation}\label{eq:A_shift}
\begin{split}
\mathcal{B}_\lambda^{h,k}(D_k^{h}u,D_k^hu)=\int_{\R^2}\Bigl(& \left.\frac{d_2}{d_1}\right|_{(x,y)+h\e_k}
       \,\left|  \p_x  D_k^hu\right|^2+\left.\frac{d_1}{d_2}\right|_{(x,y)+h\e_k}
       \,\left|  \p_y  D_k^hu\right|^2\\
       &+\lambda
       \,\left|   D_k^hu\right|^2\Bigr)dxdy
\end{split}
\end{equation}
\begin{equation}\label{eq:C_form}
\begin{split}
\mathcal{C}^{h,k}(u,D_k^hu)=\int_{\R^2}\!\Bigl(&D_k^h\left(  \frac{d_2}{d_1}\right)\,\p_x  u
       \,  \p_x  D_k^h u+D_k^h\left(  \frac{d_1}{d_2}\right)\,\p_y  u
       \,  \p_y  D_k^h u\Bigr)\,dxdy .
\end{split}
\end{equation}
We note that $\mathcal{B}_\lambda^{h,k}$ is exactly the same sesquilinear form
  as $\mathcal{B}_\lambda$ but with coefficients shifted by $h\e_k$; since
  Lemma~\ref{lemma:bounds1}(3) holds at every point,
  \eqref{eq:coercivity_B} holds uniformly for
  $\mathcal{B}_\lambda^{h,k}$ and hence,
  \begin{equation}\label{eq:coercivity_Ahk}
       \Re\, \mathcal{B}_\lambda^{h,k}(v,v) \ge C_2'\|v\|_{H^1}^2
    \qquad\forall\,v\in H^1(\R^2).
  \end{equation}

  Testing~\eqref{eq:weak_g} with $v=D_k^{-h}(D_k^hu)$ and using
  \eqref{eq:A_decompos} directly (without transforming the right-hand
  side), we obtain
  \begin{equation}\label{eq:Ahk_eq}
    \mathcal{B}_\lambda^{h,k}(D_k^hu,D_k^hu)
    = -\bigl(g,\,D_k^{-h}D_k^hu\bigr)_{\R^2} - \mathcal{C}^{h,k}(u,D_k^hu).
  \end{equation}
  Thus, by \eqref{eq:coercivity_Ahk},
  \begin{equation}
  \begin{split}\label{eq:bounds_C_hk}
    C_2'\|D^h_ku\|_{H^1(\R^2)}^2
    &\leq \Re\,\mathcal{B}_\lambda^{h,k}(D_k^hu,D_k^hu)\\
    &\leq \bigl|\bigl(g,\,D_k^{-h}D_k^hu\bigr)_{\R^2}\bigr| + |\mathcal{C}^{h,k}(u,D_k^hu)|.
  \end{split}
  \end{equation}
  For the first term on the right-hand side, since $D_k^hu\in H^1(\R^2)$
  for each fixed $h$, \cite[Lemma~9.48]{renardy} gives
  $\|D_k^{-h}(D_k^hu)\|_{L^2(\R^2)}\leq\|D_k^hu\|_{H^1(\R^2)}$, so by
  Cauchy--Schwarz and~\eqref{eq:bound_g},
  \begin{equation}\label{eq:bound_source}
     \bigl|\bigl(g,\,D_k^{-h}D_k^hu\bigr)_{\R^2}\bigr|
     \leq \|g\|_{L^2(\R^2)}\|D_k^hu\|_{H^1(\R^2)}
     \leq C_g\|f\|_{L^2(\R^2)}\|D_k^hu\|_{H^1(\R^2)},
  \end{equation}
  with $C_g$ the constant in~\eqref{eq:bound_g}.
  For $\mathcal{C}^{h,k}$, we note that the coefficients $d_2/d_1$ and
  $d_1/d_2$ are in $C^\infty(\R^2)$ with bounded derivatives --- they are
  built from $\sigma_1$ and $\sigma_2$ alone --- hence
  $|D_k^h(d_2/d_1)|$ and $|D_k^h(d_1/d_2)|$ are bounded
  uniformly in $h$ by some positive constant $\beta$. Thus, with the
  Cauchy--Schwarz inequality and \eqref{eq:bound_u_les_f} we obtain
  \begin{equation}\label{eq:bound_Chk_beta}
      \begin{split}
       |\mathcal{C}^{h,k}(u,D^h_ku)| & \leq  2\beta \|u\|_{H^1(\R^2)}\|D_k^hu\|_{H^1(\R^2)}\\
       &\leq 2\beta C\|f\|_{L^2(\R^2)}\|D_k^hu\|_{H^1(\R^2)}.\\
      \end{split}
  \end{equation}

From \eqref{eq:bounds_C_hk}, \eqref{eq:bound_source} and \eqref{eq:bound_Chk_beta} we have
\begin{equation}
    C_2'\|D^h_ku\|^2_{H^1(\R^2)}\leq (C_g+2\beta C)\|f\|_{L^2(\R^2)}\|D_k^hu\|_{H^1(\R^2)},
\end{equation}
thus,
\begin{equation}
   \|D^h_ku\|_{H^1(\R^2)} \leq \tilde{C}\|f\|_{L^2(\R^2)}.
\end{equation}
Finally, the inequality
\begin{equation}
    \begin{split}
        \left\|D_k^h\p_{x_k} u\right\|_{L^2(\R^2)}=\left\|\p_{x_k} D_k^hu\right\|_{L^2(\R^2)}\leq \|D^h_ku\|_{H^1(\R^2)}\leq \tilde C\|f\|_{L^2(\R^2)}
    \end{split}
\end{equation}
and \cite[Lemma~9.49]{renardy} imply that $\left\|\p_{x_k}^2 u\right\|_{L^2(\R^2)}\leq \tilde C\|f\|_{L^2(\R^2)}$. In this way, we can prove that the other second-order partial weak derivatives satisfy the same estimate, hence $u\in H^2(\R^2)$ and $\|u\|_{H^2(\R^2)}\leq C\|f\|_{L^2(\R^2)}$ for some positive constant $C$. We have shown that for every $f\in L^2(\R^2)$ the unique weak solution $u\in H^1(\R^2)$ of $\mathcal{A}_{\R^2,w_0}(u,v)=(Jf,v)_{\R^2}$ actually belongs to $H^2(\R^2)$ and satisfies $\|u\|_{H^2}\le C\|f\|_{L^2}$ for some constant $C>0$ independent of $f$. Moreover, testing the weak formulation against $v\in C_c^\infty(\R^2)$ and integrating by parts yields the strong equation
\[
-c^2\widetilde{\Delta}u - w_0 u = f \qquad \text{in } L^2(\R^2).
\]
Hence the operator $-c^2\widetilde{\Delta}-w_0 I$ maps $H^2(\R^2)$ onto $L^2(\R^2)$ and is bounded below: $\|u\|_{H^2}\le C\|(-c^2\widetilde{\Delta}-w_0 I)u\|_{L^2}$. Consequently, $-c^2\widetilde{\Delta}-w_0 I$ is a bijection from $H^2(\R^2)$ onto $L^2(\R^2)$ with bounded inverse, which means that $w_0$ belongs to the resolvent set of $-c^2\widetilde{\Delta}$. Thus $\rho(-c^2\widetilde{\Delta})\neq\emptyset$.
\end{proof}

\bibliographystyle{elsarticle-num}
\bibliography{references}

\end{document}